\documentclass[12pt,reqno]{article}
\usepackage[english]{babel}
\usepackage{graphicx}
\usepackage[centertags]{amsmath}
\usepackage{amsfonts,amssymb,amsthm}
\usepackage{hyperref}
\usepackage[utf8]{inputenc} 
\usepackage{mathtools} 
\usepackage{color}
\usepackage{anysize}
\marginsize{3.0cm}{3.0cm}{2.2cm}{2.2cm}

\let\OLDthebibliography\thebibliography
\renewcommand\thebibliography[1]{
  \OLDthebibliography{#1}
  \setlength{\parskip}{0pt}
  \setlength{\itemsep}{0pt plus 0.3ex} }

\numberwithin{equation}{section}

\theoremstyle{plain}
\newtheorem{theorem}{Theorem}[section]

\newtheorem{lemma}[theorem]{Lemma}
\newtheorem{corollary}[theorem]{Corollary}

\theoremstyle{definition}
\newtheorem{remarkbase}[theorem]{Remark}
\newenvironment{remark}{\pushQED{\qed}\remarkbase}{\popQED\endremarkbase}

\def\notina[#1]#2{\begingroup\def\thefootnote{\fnsymbol{footnote}}\footnote[#1]{#2}\endgroup}

\newcommand{\N}{{\mathbb N}}
\newcommand{\R}{{\mathbb R}}
\newcommand{\C}{{\mathbb C}}
\newcommand{\Z}{\mathbb Z}

\newcommand{\T}{{\mathbb T}}

\newcommand{\mB}{\mathcal{B}}

\newcommand{\mD}{\mathcal{D}}

\newcommand{\mF}{\mathcal{F}}

\newcommand{\mK}{\mathcal{K}}

\newcommand{\mR}{\mathcal{R}}
\newcommand{\mP}{\mathcal{P}}

\newcommand{\mM}{\mathcal{M}}

\renewcommand{\a}{\alpha}
\renewcommand{\b}{\beta}

\renewcommand{\d}{\delta}

\newcommand{\e}{\varepsilon}
\newcommand{\ph}{\varphi}

\newcommand{\Om}{\Omega}

\renewcommand{\th}{\vartheta}

\renewcommand{\Re}{\mathrm{Re}\,}
\newcommand{\sign}{\mathrm{sign}}

\newcommand{\gr}{\nabla}

\newcommand{\la}{\langle}
\newcommand{\ra}{\rangle}

\newcommand{\pa}{\partial}
\newcommand{\Lm}{\Lambda}

\newcommand{\baru}{\overline{u}}

\newcommand{\be}{\begin{equation}}
\newcommand{\ee}{\end{equation}}
\newcommand{\bal}{\begin{align}} 
\newcommand{\eal}{\end{align}} 

\title{\Large{\textbf{
On the 
existence time for the Kirchhoff equation with periodic boundary conditions}}}

\author{\small{Pietro Baldi and Emanuele Haus}} 

\date{} 

\begin{document}

\maketitle

\noindent
\emph{Abstract}.\  We consider the Cauchy problem for the Kirchhoff equation on $\T^d$
with initial data of small amplitude $\e$ in Sobolev class. 
We prove a lower bound $\e^{-4}$ for the existence time, 
which improves the bound $\e^{-2}$ given by the standard local theory. 
The proof relies on a normal form transformation, 
preceded by a nonlinear transformation that diagonalizes the operator 
at the highest order, 
which is needed because of the quasilinear nature of the equation.

\medskip

\noindent
\emph{Keywords}.\ 
Kirchhoff equation, quasilinear wave equations, 
Cauchy problems, normal forms, quasilinear normal forms.

\noindent
\emph{MSC2010}:  
35L72, 
35L15, 
35Q74, 
37J40, 
70K45. 


\bigskip
\bigskip

\noindent
\emph{Contents.} 
\ref{sec:1} Introduction ---
\ref{sec:Lin} Linear transformations ---
\ref{sec:Bubaru} Diagonalization of the order one ---
\ref{sec:NF} Normal form transformation ---
\ref{sec:proof} Proof of Theorem \ref{thm:main}.

\section{Introduction}  
\label{sec:1}

This paper deals with an old open problem, concerning the global wellposedness 
of the Kirchhoff equation 
\begin{equation} \label{K Om}
\pa_{tt} u - \Big( 1 + \int_{\Om} |\gr u|^2 \, dx \Big) \Delta u = 0
\end{equation}
with periodic boundary conditions $\Om = \T^d$ 
or Dirichlet boundary conditions $u|_{\pa \Om} = 0$ 
on a bounded domain $\Om \subset \R^d$. 
In 1940 Bernstein \cite{Bernstein 1940} proved that, in the 1-dimensional case
$\Om = [0,\pi]$ with zero boundary conditions $u(t,0) = u(t,\pi) = 0$, 
the Cauchy problem for \eqref{K Om} with initial data 
\begin{equation} \label{init data Om}
u(0,x) = \a(x), \quad 
\pa_t u(0,x) = \b(x)
\end{equation}
is globally wellposed for $(\a,\b)$ analytic, 
and locally wellposed for $(\a,\b)$ in the Sobolev space $H^2 \times H^1$.  
Later on, these results have been extended to higher dimension, 
also including the periodic setting $\Om = \T^d$, 
proving global wellposedness in larger spaces containing the analytic functions, 
and local wellposedness in the Sobolev space $H^{\frac32} \times H^{\frac12}$, 
with existence time $T \sim (\| \a \|_{\frac32} + \| \b \|_{\frac12})^{-2}$, 
see Section \ref{sec:literature} for a short overview.
Nonetheless, the basic question about the existence time for the Cauchy problem 
with $C^\infty$ data, even of small amplitude, is still open.  
In particular, it is still not known whether the maximal existence time is finite or infinite 
(notice that the quasilinear wave equation $u_{tt} - (1 + u_x^2) u_{xx} = 0$ on the circle $\T$, 
which looks like \eqref{K1} in one dimension without the integral sign,
has a finite blowup time $T \sim (\| \a \|_{C^2} + \| \b \|_{C^1})^{-2}$, 
as proved by Klainerman and Majda \cite{Klainerman Majda 1980}).

In this paper we prove that in the periodic setting $\Om = \T^d$, $d \geq 1$, 
for small amplitude initial data 
$(\a,\b) \in H^{\frac32} \times H^{\frac12}$ if $d = 1$, 
or $(\a,\b) \in H^2 \times H^1$ if $d \geq 2$, 
the existence time is at least $T \sim (\| \a \| + \| \b \|)^{-4}$
(Theorem \ref{thm:main}),
which is longer than the time $(\| \a \| + \| \b \|)^{-2}$ 
provided by the classical local theory.  
The same result also holds in the case of zero Dirichlet boundary conditions 
on the cube $\Om = [0,\pi]^d$ (Remark \ref{rem:Dirichlet}).
To give a precise statement of our main result, we first introduce some notation.

\bigskip


On the torus $\T^d$, it is not restrictive to set the problem in the space of functions with zero average in space, for the following reason.
Given initial data $\a(x), \b(x)$, we split both them and the unknown $u(t,x)$ into the sum of a zero-mean function and the average term, 
\[
\a(x) = \a_0 + \tilde \a(x), \quad 
\b(x) = \b_0 + \tilde \b(x), \quad
u(t,x) = u_0(t) + \tilde u(t,x),
\]
where 
\[
\int_{\T^d} \tilde\a(x) \, dx = 0, \quad
\int_{\T^d} \tilde\b(x) \, dx = 0, \quad
\int_{\T^d} \tilde u(t,x) \, dx = 0 \quad \forall t.
\]
Then the Cauchy problem 
\begin{equation} \label{K1}
\pa_{tt} u -  \Big( 1 + \int_{\T^d} |\gr u|^2 \, dx \Big) \Delta u = 0, \quad 
u(0,x) = \a(x), \quad 
\pa_t u(0,x) = \b(x)
\end{equation}
splits into two distinct, uncoupled Cauchy problems: 
one is the problem for the average $u_0(t)$, which is
\[
u_0''(t) = 0, \quad u_0(0) = \a_0, \quad u_0'(0) = \b_0
\]
and has the unique solution $u_0(t) = \a_0 + \b_0 t$; 
the other one is the problem for the zero-mean component $\tilde u(t,x)$, which is 
\[
\pa_{tt} \tilde u - \Big( 1 + \int_{\T^d} |\gr \tilde u|^2 \, dx \Big) \Delta \tilde u = 0, 
\quad
\tilde u(0,x) = \tilde \a(x), \quad
\pa_t \tilde u(0,x) = \tilde \b(x).
\]
Thus one has to study the Cauchy problem for the zero-mean unknown $\tilde u(t,x)$ 
with zero-mean initial data $\tilde \a(x), \tilde \b(x)$; 
this means to study \eqref{K1} in the class of functions with zero average in $x$.

For any real $s \geq 0$, we consider the Sobolev space of zero-mean functions
\begin{align} \label{def:Hs}
H^s_0(\T^d,\C) & := \Big\{ u(x) = \sum_{j \in \Z^d \setminus \{ 0 \} } u_j e^{ij\cdot x} : 
u_j \in \C, \ \| u \|_s < \infty \Big\}, 
\\
\| u \|_s^2 & := \sum_{j \in \Z^d \setminus \{ 0 \}} |u_j|^2 |j|^{2s},
\notag
\end{align}
and its subspace 
\begin{equation*} 
H^s_0(\T^d,\R) := \{ u \in H^s_0(\T^d,\C) : u(x) \in \R \}
\end{equation*}
of real-valued functions. 
For $s=0$, we write $L^2_0$ instead of $H^0_0$ the space of square-integrable functions with zero average.

The main result of the paper is the following theorem. 

\begin{theorem} \label{thm:main}
For $d \in \N$, let
\begin{equation} \label{def m0}
m_0 = 1 \quad \text{if} \ \ d = 1, 
\qquad 
m_0 = \frac32 \quad \text{if} \ \ d \geq 2.
\end{equation}
There exist universal constants $\e_0, C, C_1 > 0$ with the following properties. 

If $(\a, \b) \in H^{m_0 + \frac12}_0(\T^d,\R) \times H^{m_0 - \frac12}_0(\T^d,\R)$
with 
\[
\e := \| \a \|_{m_0 + \frac12} + \| \b \|_{m_0 - \frac12} \leq \e_0, 
\]
then the Cauchy problem \eqref{K1} has a unique solution 
$u \in C^0([0,T], H^{m_0 + \frac12}_0(\T^d,\R)) \cap C^1([0,T], H^{m_0 - \frac12}_0(\T^d,\R))$ 
on the time interval $[0,T]$, where 
\[
T = \frac{C_1}{\e^4},
\]
and 
\[
\max_{t \in [0,T]} ( \| u(t) \|_{m_0 + \frac12} + \| \pa_t u(t) \|_{m_0 - \frac12} ) 
\leq C \e.
\]
If, in addition, $(\a,\b) \in H^{s + \frac12}_0(\T,\R) \times H^{s - \frac12}_0(\T,\R)$
for some $s \geq m_0$, then $u$ belongs to 
$C^0([0,T], H^{s + \frac12}_0(\T^d,\R)) \cap C^1([0,T], H^{s - \frac12}_0(\T^d,\R))$, with
\begin{equation} \label{SUR} 
\max_{t \in [0,T]} ( \| u(t) \|_{s + \frac12} + \| \pa_t u(t) \|_{s - \frac12} ) 
\leq C ( \| \a \|_{s + \frac12} + \| \b \|_{s - \frac12} ).
\end{equation}
\end{theorem}

\begin{remark}[Evolution of higher norms]
\label{rem:Evo}
The constant $C$ in \eqref{SUR} does not depend on $s$. 
This unusual property is a consequence of the special structure of the Kirchhoff equation: 
if $u$ is a solution of \eqref{K Om}, 
then $u$ also solves the linear wave equation with time-dependent coefficient
$\pa_{tt} u - a(t) \Delta u = 0$, with $a(t) = 1 + \int_{\Om} |\gr u|^2 \, dx$, 
and therefore $v := |D_x|^s u$ also solves $\pa_{tt} v - a(t) \Delta v = 0$.
\end{remark}

\begin{remark}[Why $m_0$ in \eqref{def m0} is different in dimension $d=1$ and $d \geq 2$]
\label{rem:Why m0}
The proof of Theorem \ref{thm:main} is based on a normal form transformation.
In the construction of such a normal form, one encounters the differences 
of the linear eigenvalues $|j|$, $j \in \Z^d$, 
as denominators of the transformation coefficients (see \eqref{fix A12}-\eqref{fix C12}).   
On the 1-dimensional torus $\T$, the difference 
$||j|-|k||$ is either zero or $\geq 1$, 
while on 
$\T^d$, $d \geq 2$, 
the differences $||j|-|k|| = | \sqrt{j_1^2 + \ldots + j_d^2} - \sqrt{k_1^2 + \ldots + k_d^2}|$ 
accumulate to zero, with lower bounds $||j|-|k|| \geq \frac{1}{|j| + |k|}$. 
This is the reason for the different regularity threshold we obtain in dimension 1 or higher. 
\end{remark}

\begin{remark}[Dirichlet boundary conditions on the cube] 
\label{rem:Dirichlet}
Theorem \ref{thm:main} im\-me\-di\-ate\-ly implies a similar result for the Cauchy problem 
with zero Dirichlet boundary conditions on the cube $\Om := [0,\pi]^d$. 
Given any function $u : \Om \to \R$, 
let 
\[
U : [-\pi, \pi]^d \to \R, \quad  
U(x) := \sign(x_1 x_2 \cdots x_d) u(|x_1|, \ldots, |x_d|)
\]
be its extension by odd reflection, 
and let $u_{ext} : \T^d \to \R$ be the periodic extension of $U$.

A function $u$ belongs to $H^s(\Om)$, $s = 1$ or $s = 2$ 
(i.e.\ the weak partial derivatives of order $\leq s$ belong to $L^2(\Om)$)
with Dirichlet boundary condition $u = 0$ on the boundary $\pa \Om$ 
if and only if (see, e.g., \cite{Fu}, \cite{Arosio Panizzi 1996}) 
$u$ belongs to the domain $V_s(\Om)$ of the fractional Laplacian $(-\Delta)^{s/2}$ 
on $\Om$ with zero Dirichlet boundary conditions (a spectrally defined Sobolev space). 
In such a case, the extension $u_{ext}$ belongs to the Sobolev space $H^s_0(\T^d)$ 
defined in \eqref{def:Hs}. 
Hence, for initial data $\a \in H^2(\Om)$, $\b \in H^1(\Om)$ 
with $\a = \b = 0$ on $\pa \Om$, one consider the periodic odd extensions 
$\a_{ext} \in H^2_0(\T^d)$, $\b_{ext} \in H^1_0(\T^d)$, 
and Theorem \ref{thm:main} applies. 

In dimension $d = 1$, Theorem \ref{thm:main} requires less regularity,  
and it is sufficient that $\a \in V_{\frac32}(0,\pi)$ and $\b \in V_{\frac12}(0,\pi)$. 
One has $\a \in V_{\frac32}(0,\pi)$ if and only if $\a$ belongs to the fractional Sobolev space 
$H^{\frac32}(0,\pi)$ on the interval, with $\a(0) = \a(\pi) = 0$, 
while $\b \in V_{\frac12}(0,\pi)$ if and only if $\b \in H^{\frac12}(0,\pi)$ with
$\int_0^\pi \frac{|\b(x)|^2}{x(\pi-x)} \, dx < \infty$
(see \cite{Fu}, \cite{Arosio Panizzi 1996}).
\end{remark}

\subsection{Strategy of the proof}

Since the problem is set on the torus $\T^d$, which is a compact manifold, 
no dispersive estimates are available to study the long-time dynamics, 
and the main point is the analysis of the resonances, 
for which the key tool is the normal form theory.

The main difficulty in the application of the normal form theory to the Kirchhoff equation
is due to the fact that it is a \emph{quasilinear} PDE.
Let us explain this point in more detail. 
The Kirchhoff equation has the Hamiltonian structure
\be \label{p1}
\begin{cases} 
\pa_t u = \gr_v H(u,v) = v, \\ 
\pa_t v = - \gr_u H(u,v) = \Big( 1 + \int_{\T^d} |\gr u|^2 dx \Big) \Delta u,
\end{cases}
\ee
where the Hamiltonian is
\be \label{K2} 
H(u,v) = \frac12 \int_{\T^d} v^2 dx 
+ \frac12 \int_{\T^d} |\gr u|^2 dx 
+ \Big( \frac12 \int_{\T^d} |\gr u|^2 dx \Big)^2,
\ee
and $\gr_u H$, $\gr_v H$ are the gradients with respect to the real scalar product 
\begin{equation} \label{lara}
\la f,g \ra := \int_{\T^d} f(x) g(x) \, dx \quad \forall f,g \in L^2(\T^d, \R),
\end{equation}
namely $H'(u,v)[f,g] = \la \gr_u H(u,v) , f \ra + \la \gr_v H(u,v) , g \ra$
for all $u,v,f,g$.
As a consequence, 
the first natural attempt is trying to construct the Birkhoff normal form,
using close-to-identity, symplectic transformations 
that are the time one flow of auxiliary Hamiltonians, 
with the goal of removing the nonresonant terms from the Hamiltonian \eqref{K2},
proceeding step by step with respect to the homogeneity orders. 
When one calculates (at least formally) the first step of this procedure, 
one finds a transformation $\Phi$ that is 
bounded on a ball of $H^s(\T^d,\R) \times H^{s-1}(\T^d,\R)$ around the origin, 
but it is not close to the identity as a bounded operator, 
in the sense that $\| \Phi(u,v) - (u,v) \|_{H^s \times H^{s-1}}$ 
is not $\lesssim \| (u,v) \|_{H^s \times H^{s-1}}^3$, 
as one needs for the application of the Birkhoff normal form method. 
Hence the transformed Hamiltonian $H(\Phi(u,v))$ cannot be Taylor expanded
in homogeneous orders without paying a loss of derivative, 
and the Birkhoff normal form procedure fails. 
This is ultimately a consequence of the quasilinear nature of the Kirchhoff equation.
Also, even working with more general close-to-identity transformations of vector fields, 
not necessarily preserving the Hamiltonian structure, 
the direct application of the Poincar\'e normal form procedure
encounters the same obstacle.


Thus, one has to look at the equation more carefully, 
distinguishing some terms that are harmless 
and some other terms that are responsible for the failure of the normal form construction.
To this aim, it is convenient to introduce symmetrized complex coordinates (see Section \ref{sec:Lin}), 
so that the linear wave operator becomes diagonal, 
and system \eqref{p1} becomes 
(see \eqref{syst uv})
\begin{equation} \label{u baru}
\begin{cases}
\pa_t u = - i \Lm u 
- \frac{i}{4} \la \Lm (u + \overline{u}) , u + \overline{u} \ra \Lm (u + \overline{u}),
\\
\pa_t \overline{u} = i \Lm \overline{u} 
+ \frac{i}{4} \la \Lm (u + \overline{u}) , u + \overline{u} \ra \Lm (u + \overline{u}),
\end{cases}
\end{equation}
where $\overline{u}$ is the complex conjugate of $u$, 
$\Lm := |D_x|$ is the Fourier multiplier of symbol $|\xi|$,
and $\la f , g \ra := \int_{\T^d} f(x) g(x) \, dx$ is the same as in \eqref{lara}, 
even for complex-valued functions $f,g$. 
We note that the cubic nonlinearity in \eqref{u baru} 
already has a ``paralinear'' structure, in the sense that, 
for all functions $u,v,h$, all $s \geq 0$, one has 
\[
\| \la \Lm u,v \ra \Lm h \|_s = | \la \Lm u,v \ra | \, \| h \|_{s+1} 
\leq \| u \|_{\frac12} \| v \|_{\frac12} \| h \|_{s+1}.
\]
Hence \eqref{u baru} can be interpreted as a linear system whose operator coefficients 
depend on $(u, \baru)$, namely
\begin{equation} \label{u baru Q}
\pa_t \begin{pmatrix} u \\ \overline{u} \end{pmatrix}
= \begin{pmatrix} - A(u,\baru) & - B(u,\baru) \\ 
B(u, \baru) & A(u, \baru) \end{pmatrix} 
\begin{pmatrix} u \\ \baru \end{pmatrix},
\end{equation}
where
\[
B(u, \baru) = \frac{i}{4} \la \Lm (u + \overline{u}) , u + \overline{u} \ra \Lm, 
\quad 
A(u, \baru) = i \Lm + B(u, \baru).
\]
Since our goal is the analysis of the existence time of the solutions, 
we calculate the time derivative $\pa_t (\| u \|_s^2)$ of the Sobolev norms 
and observe that the diagonal terms $A(u, \baru)$ give a zero contribution, 
while the off-diagonal terms $B(u,\baru)$, which couple $u$ with $\baru$, 
give terms that are $\leq 2 \| u \|_{\frac12}^2 \| u \|_{s+\frac12}^2$ only.
Thus, on the one hand, this energy estimate has a loss of half a derivative 
and cannot be used for the existence theory; 
on the other hand, this observation suggests that $A(u,\baru)$ can be left untouched 
by the normal form transformation. 

Hence the next natural attempt is the construction of a ``partial'' 
normal form transformation $\Phi$ that eliminates the cubic nonresonant terms only from 
$B(u,\baru)$ and does not modify $A(u,\baru)$. 
Indeed, such a transformation exists, it is bounded, and, 
unlike the full normal form, 
is close to the identity as a bounded transformation, namely  
$\| \Phi(u,\baru) - (u,\baru) \|_{H^s \times H^s} 
\lesssim \| (u,\baru) \|_{H^s \times H^s}^3$. 
Moreover, the cubic resonant terms of $B(u,\baru)$ that remain in the transformed system 
give zero contribution to the energy estimate. 
However, the transformed system contains unbounded off-diagonal terms 
of quintic and higher homogeneity order, 
which produce in the energy estimate
the same loss of half a derivative as above.

At this point it becomes clear that one has to eliminate the off-diagonal unbounded terms 
\emph{before} the normal form construction. 
This is at the base of the method developed by Delort in \cite{Delort 2009}, \cite{Delort 2012} 
to construct a normal form for quasilinear Klein-Gordon equations on the circle.
Roughly speaking, such a method consists in paralinearizing the equation, 
diagonalizing its principal symbol, so that one can obtain quasilinear 
energy estimates, and then starting with the normal form procedure.
Further developments of this approach can be found in Berti and Delort 
\cite{Berti Delort} about gravity-capillary water waves equations on $\T$. 

The off-diagonal unbounded terms of \eqref{u baru} 
are eliminated in Section \ref{sec:Bubaru},
where we construct a nonlinear bounded transformation $\Phi^{(3)}$ 
that conjugates system \eqref{u baru} to a new system (see \eqref{syst 6 dic}) 
of the form
\begin{equation} \label{u baru P}
\begin{cases}
\pa_t u = - i \sqrt{1 + 2 P(u,\baru)} \, \Lm u
+ \dfrac{i}{4(1+2 P(u,\baru))} 
\Big( \la \Lm \baru, \Lm \baru \ra - \la \Lm u , \Lm u \ra \Big) \baru,
\\
\pa_t \baru = i \sqrt{1 + 2 P(u,\baru)} \, \Lm \baru
+ \dfrac{i}{4(1+2 P(u,\baru))} 
\Big( \la \Lm \baru, \Lm \baru \ra - \la \Lm u , \Lm u \ra \Big) u,
\end{cases}
\end{equation}
where $P(u, \baru)$ is a real, nonnegative function of time only, 
defined as $P(u, \baru) = \ph(\frac14 \la \Lm(u+\baru), u+\baru \ra)$, 
and $\ph$ is the inverse of the real map $x \mapsto x \sqrt{1+2x}$, $x \geq 0$.
System \eqref{u baru P} still has the structure \eqref{u baru Q}, 
with the improvement that the off-diagonal part $B(u,\baru)$ 
is now a \emph{bounded} operator, satisfying 
\[
\| B(u, \baru) h \|_s \leq \| u \|_1^2 \| h \|_s
\]
for all $s \geq 0$, all $u,h$. 
Thanks to the special structure of the Kirchhoff equation, 
and in particular to the lower bound 
$\frac14 \la \Lm(u+\baru), u+\baru \ra = \int_{\T^d} (\Re (\Lm^{\frac12} u) )^2 \, dx \geq 0$,
the transformation $\Phi^{(3)}$ is \emph{global}, 
namely it is defined for all $u \in H^1_0(\T^d, \C)$, 
and not only for small $u$. 
In \eqref{u baru} the off-diagonal term is an operator of order one 
with coefficient $\la \Lm(u+\baru), u+\baru \ra$ 
defined for $u \in H^{\frac12}_0(\T^d,\C)$, 
while, after $\Phi^{(3)}$, the new off-diagonal term in \eqref{u baru P} 
is an operator of order zero where the coefficient 
$( \la \Lm \baru, \Lm \baru \ra - \la \Lm u, \Lm u \ra)$ 
is defined for $u \in H^1_0(\T^d,\C)$. 
Thus the price to pay for removing the unbounded off-diagonal terms 
is an increase of $\frac12$ in the regularity threshold for $u$
(as if we had integrated by parts).

We remark that, reparametrizing the time variable, 
the coefficient $\sqrt{1 + 2 P(u,\baru)}$ of the diagonal part in \eqref{u baru P}
could be normalized to 1; however, this is not needed to prove our result, 
because these coefficients are independent of $x$, 
and therefore the (unbounded) diagonal terms cancel out in the energy estimate. 

In Section \ref{sec:NF} we perform one step of normal form. 
It is a ``partial'' normal form because it does not modify the harmless cubic diagonal terms. 
The construction involves the differences $|j| - |k|$, $j, k \in \Z^d$, $j \neq k$,
as denominators, which accumulate to zero in dimension $d \geq 2$. 
This produces the different regularity thresholds $m_0$ in Theorem \ref{thm:main}, 
see Remark \ref{rem:Why m0}.
The normal form transformation $\Phi^{(4)}$ is a bounded cubic correction of the identity map,
and the off-diagonal terms of the transformed system \eqref{def X+}, \eqref{X+ ter} 
remain bounded (unlike in the discussion above). 
The resonant cubic terms that remain after $\Phi^{(4)}$ create a nonlinear interaction 
between all Fourier coefficients $u_j(t)$ with Fourier modes $j \in \Z^d$  
on a sphere $|j| = $ constant, while any two Fourier coefficients $u_j(t), u_k(t)$ 
with $|j| \neq |k|$ are uncoupled at the cubic homogeneity order. 
This, together with the conservation of the Hamiltonian, implies that there is no growth 
of Sobolev norms at the cubic homogeneity order. 
Therefore all the possible nonlinear effects of growth of Sobolev norms 
come from the terms of quintic and higher homogeneity order. 
This leads to the improved energy estimate (see \eqref{en est X+})
\[
\pa_t ( \| u(t) \|_s^2) \leq C \| u(t) \|_{m_0}^4 \| u(t) \|_s^2 
\]
for the transformed system, 
whence we deduce that the lifespan of the solutions of the original Cauchy problem \eqref{K1} 
is $T \sim (\| \a \|_{s+\frac12} + \| \b \|_{s-\frac12})^{-4}$. 

Preliminary further calculations suggest that, 
after performing the next step of normal form 
to remove the off-diagonal nonresonant quintic terms, 
some remaining quintic resonant terms could produce 
a nonlinear interaction between modes $|j| \neq |k|$, 
so that, in principle, a transfer of energy from low to high Fourier modes, 
and a growth of Sobolev norms 
(as in \cite{CKSTT 2010}, \cite{Guardia Kaloshin 2015},  
\cite{Haus Procesi 2015}, \cite{Guardia Haus Procesi 2016} 
for the semilinear Schr\"odinger equation on $\T^2$)
cannot be excluded.
The analysis of the quintic order is the objective of a further investigation.

As a final comment, we observe that the general strategy developed in 
\cite{Delort 2009}, \cite{Delort 2012}, \cite{Berti Delort} 
and also adopted in the present paper 
has a strong analogy with the technique developed for KAM theory for quasilinear PDEs 
in 
\cite{BBM Airy},
\cite{BBM 2016 KdV}, 
\cite{Feola Procesi 2015}, 
\cite{Berti Montalto}, 
\cite{Baldi Berti Haus Montalto}:
the first part of these methods uses pseudo-differential or paradifferential calculus 
to reduce the linearized or paralinearized operator 
to some more convenient diagonal form up to a sufficiently smoothing remainder, 
and it is a reduction with respect to the order of differentiation; 
then the second part uses normal forms or KAM reducibility schemes
to reduce the size of the nonconstant remainders in the operator. 
In short: first reduce in $|D_x|$, then in $\e$.

\subsection{Reversible Hamiltonian structure and prime integrals}
\label{sec:strutture} 

In this section we make some observations 
about the structure of the Kirchhoff equation. 
We do not use them directly in the proof of Theorem \ref{thm:main}, 
but they could be interesting \emph{per se}. 

As is well-known, the Kirchhoff equation has a Hamiltonian structure, 
which is \eqref{p1}-\eqref{K2}. 
Also, since the Hamiltonian \eqref{K2} is even in $v$, namely $H(u,-v) = H(u,v)$, 
the Hamiltonian vector field $X(u,v) = (\gr_v H(u,v), - \gr_u H(u,v))$ 
satisfies $X \circ S + S \circ X = 0$, where $S$ is the involution 
$S (u,v) = (u, -v)$. Therefore system \eqref{p1} is time-reversible with respect to $S$, 
which simply means that if $u(t,x)$ is a solution of \eqref{K Om}, 
then $u(-t,x)$ is also a solution of the same equation.  

Another observation is that the space of functions $u(t,x) = u(t, -x)$ that are even in $x$ 
is an invariant subspace for the Kirchhoff equation, 
as well as the space of odd functions $u(t,x) = -u(t, -x)$. 
The Fourier support is also invariant for the flow:  
since the Kirchhoff equation for $u(t,x) = \sum_{j \in \Z^d} u_j(t) e^{ij \cdot x}$ is 
the system of equations
\begin{equation} \label{ODE}
u_j'' + |j|^2 u_j \Big( 1 + \sum_{k \in \Z^d} |k|^2 |u_k|^2 \Big) = 0 
\quad \forall j \in \Z^d,
\end{equation}
if $u_j(0) = u_j'(0) = 0$ for some $j$, then $u_j(t) = 0$ for all $t$.
In particular, if the initial data $(\a,\b)$ have finite Fourier support, 
then the solution exists for all times, 
and a simple application of finite-dimensional KAM theory shows that 
some of them are quasi-periodic in time.

In addition to the Hamiltonian, the momentum 
\[
M = \int_{\T^d} (\pa_t u) \gr u \, dx
\]
is also a conserved quantity. 
Even more, because of the special structure of the Kirchhoff equation, 
the momentum is the sum $M = \sum_{j \in \Z^d} M_j$ 
of infinitely many prime integrals $M_j$, defined in the following way.  
If $u(t,x) = \sum_{j \in \Z^d} u_j(t) e^{ij \cdot x}$, then 
\[
M_j = \frac12 ij (u_j \pa_t u_{-j} - u_{-j} \pa_t u_j), \quad j \in \Z^d,
\]
and one has 
\[
\pa_t M_j = \frac12 ij (u_j \pa_{tt} u_{-j} - u_{-j} \pa_{tt} u_j)
= 0
\]
because each $u_j$ satisfies \eqref{ODE}. 
This observation seems to be new.
Since $M_{-j} = M_j$, only ``a half'' of these prime integrals are independent. 

The standard linear changes of coordinates of Section \ref{sec:Lin} 
preserve both the Hamiltonian and the reversible structure 
($S$ becomes $S_1 (u,\baru) = (\baru, u)$ in complex coordinates).
The nonlinear transformation of Section \ref{sec:Bubaru} is not symplectic, 
but it preserves the reversible structure, which is also preserved 
by the normal form of Section \ref{sec:NF}.

\subsection{Related literature and open questions} 
\label{sec:literature}


Equation \eqref{K Om} was introduced by Kirchhoff \cite{Kirchhoff 1876}  
(and, in one dimension, independently rediscovered in \cite{Carrier 1945} and \cite{Narasimha 1968}) 
to model the transversal oscillations of a clamped string or plate, 
taking into account nonlinear elastic effects.
The first results on the Cauchy problem \eqref{K Om}-\eqref{init data Om}
are due to Bernstein. In his 1940 pioneering paper \cite{Bernstein 1940}, 
he studied the Cauchy problem on an interval, with Dirichlet boundary conditions, 
and proved global wellposedness for analytic initial data $(\a,\b)$, 
and local wellposedness for $(\a, \b) \in H^2 \times H^1$.

After that, the research on the Kirchhoff equation has been developed in various directions,
with a different kind of results on compact domains (bounded subset $\Om \subset \R^d$ with Dirichlet boundary conditions, or periodic boundary conditions $\Om = \T^d$) 
or non compact domains ($\Om = \R^d$ or ``exterior domains'' $\Om = \R^d \setminus K$, 
with $K \subset \R^d$ compact domain).

For $\Om = \R^d$, Greenberg and Hu \cite{Greenberg Hu 1980} in dimension $d=1$ 
and D'Ancona and Spagnolo \cite{D'Ancona Spagnolo 1993} in higher dimension 
proved global wellposedness with scattering for small initial data in weighted Sobolev spaces. 
Further improvements, dealing with spectrally characterized initial data 
in larger subsets of the Sobolev spaces, 
and also including the case of exterior domains, 
have been more recently obtained, for example, 
by Yamazaki, Matsuyama and Ruzhansky, 
see \cite{Yamazaki 2004}, \cite{Matsuyama Ruzhansky 2013} and the many references therein.
For global solutions that do not scatter see \cite{Matsuyama 2006}.
Still open is the main question whether the solutions 
with small initial data in the standard (not weighted) Sobolev spaces 
$H^s(\R^d) \times H^{s-1}(\R^d)$ are globally defined. 

Another research direction regards the extension of global wellposedness, 
on both compact and non compact domains, 
to non small initial data that are in a larger space than analytic functions: 
see, for example, 
Pokhozhaev \cite{Pokhozhaev 1975}, 
Arosio and Spagnolo \cite{Arosio Spagnolo 1984}, 
Nishihara \cite{Nishihara 1984},
Manfrin \cite{Manfrin 2005}, 
Ghisi and Gobbino \cite{Ghisi Gobbino 2011}, 
and the references therein. 
Still open is the question whether the solutions with initial data of arbitrary size 
and Gevrey regularity (on any domain) are globally defined. 

On compact domains, 
dispersion, scattering and time-decay mechanisms are not available, 
and there are no results of global existence, nor of finite time blowup, 
for initial data $(\a,\b)$ of Sobolev, or $C^\infty$, or Gevrey regularity. 
The local wellposedness in the Sobolev class $H^{\frac32} \times H^{\frac12}$ 
has been proved by 
Dickey \cite{Dickey 1969}, 
Medeiros and Milla Miranda \cite{Medeiros Miranda 1987}
and Arosio and Panizzi \cite{Arosio Panizzi 1996}, 
with existence time of order $(\| \a \|_{\frac32} + \| \b \|_{\frac12})^{-2} $. 
Beyond the question about the global wellposedness for small data in Sobolev class, 
another open question concerns 
the local wellposedness in the energy space $H^1 \times L^2$ 
or in $H^s \times H^{s-1}$ for $1 < s < \frac32$. 

For more details, generalizations (degenerate Kirchhoff equations, Kirchhoff systems, 
forced and/or damped Kirchhoff equations, etc.) 
and other open questions, we refer to 
Lions \cite{Lions 1978} and the surveys of 
Arosio \cite{Arosio 1993}, 
Spagnolo \cite{Spagnolo 1994}, 
and Matsuyama and Ruzhansky \cite{Matsuyama Ruzhansky 2015}.

We also mention the recent results \cite{Baldi 2009},
\cite{Montalto 2017 KAM K forced},
\cite{Montalto 2017 linear reducibility},
\cite{Corsi Montalto 2018},
which prove the existence of time periodic or quasi-periodic solutions 
of time periodically or quasi-periodically forced Kirchhoff equations on $\T^d$,
using Nash-Moser and KAM techniques. 

\medskip

Concerning the normal form theory for quasilinear PDEs, 
we mention the pioneering work of Shatah \cite{Shatah 1985} 
on quasilinear Klein-Gordon equations on $\R^d$, 
the abstract result of Bambusi \cite{Bambusi 2005},
the aforementioned papers of Delort \cite{Delort 2009}, \cite{Delort 2012} 
on quasilinear Klein-Gordon on $\T$,
and the recent literature on water waves by 
Wu \cite{Wu 2009},
Germain, Masmoudi and Shatah \cite{Germain Masmoudi Shatah 2012},
Alazard and Delort \cite{Alazard Delort 2015},
Ionescu and Pusateri \cite{Ionescu Pusateri 2015},
Craig and Sulem \cite{Craig Sulem 2016 BUMI},
Ifrim and Tataru \cite{Ifrim Tataru 2017},
Berti and Delort \cite{Berti Delort}.
Other applications of normal form techniques 
to get long-time existence for nonlinear PDEs on compact domains
can be found in the work of Bambusi, Nekhoroshev, Gr\'ebert, Delort and Szeftel, 
see e.g.\ 
\cite{Bambusi Nekhoroshev 1998},
\cite{Bambusi Grebert 2006},
\cite{Bambusi Delort Grebert Szeftel 2007},
and Feola, Giuliani and Pasquali 
\cite{Feola Giuliani Pasquali 2018}.

\bigskip

\noindent
\textbf{Acknowledgements}. 
This research was supported by ERC under FP7, Project no.\,306414 
\emph{Hamiltonian PDEs and small divisor problems: a dynamical systems approach}
and by PRIN 2015 
\emph{Variational methods, with applications to problems in mathematical physics and geometry}.

\section{Linear transformations} \label{sec:Lin}
In this section we make two elementary, standard linear changes of variables 
to transform system \eqref{p1} into another one (see \eqref{syst uv}) 
where the linear part is diagonal, 
preserving both the real and the Hamiltonian structure of the problem.
These standard transformations are the symmetrization of the highest order 
(section \ref{subsec:symm highest})
and then the diagonalization of the linear terms
(section \ref{subsec:diag highest}).

\subsection{Symmetrization of the highest order}
\label{subsec:symm highest}
In the Sobolev spaces \eqref{def:Hs} of zero-mean functions, 
the Fourier multiplier 
\begin{equation*} 
\Lm := |D_x| : H^s_0 \to H^{s-1}_0, \quad  
e^{ij \cdot x} \mapsto |j| e^{ij \cdot x}
\end{equation*}
is invertible.   
System \eqref{p1} writes
\begin{equation} \label{1912.1}
\begin{cases} 
\pa_t u = v \\ 
\pa_t v = - ( 1 + \la \Lm u, \Lm u \ra ) \Lm^2 u,
\end{cases}
\end{equation}
where $\la \cdot , \cdot \ra$ is defined in \eqref{lara};
the Hamiltonian \eqref{K2} is
\[
H(u,v) = \frac12 \la v , v \ra + \frac12 \la \Lm u, \Lm u \ra 
+ \frac14 \la \Lm u, \Lm u \ra^2.
\]
To symmetrize the system at the highest order, 
we consider the linear, symplectic transformation 
\begin{equation} \label{1912.2}
(u,v) = \Phi^{(1)}(q,p) 
= ( \Lm^{-\frac12} q , \Lm^{\frac12} p). 
\end{equation}
System \eqref{1912.1} becomes 
\begin{equation} \label{1912.3}
\begin{cases}
\pa_t q = \Lm p \\ 
\pa_t p = - ( 1 + \la \Lm^{\frac12} q, \Lm^{\frac12} q \ra ) \Lm q,
\end{cases}
\end{equation}
which is the Hamiltonian system $\pa_t (q,p) = J \gr H^{(1)}(q,p)$ 
with Hamiltonian $H^{(1)} = H \circ \Phi^{(1)}$, namely
\begin{equation} \label{def H(1)}
H^{(1)}(q,p) 
= \frac12 \la \Lm^{\frac12} p, \Lm^{\frac12} p \ra 
+ \frac12 \la \Lm^{\frac12} q, \Lm^{\frac12} q \ra 
+ \frac14 \la \Lm^{\frac12} q, \Lm^{\frac12} q \ra^2, 
\quad 
J := \begin{pmatrix} 
0 & I \\ 
- I & 0 \end{pmatrix}.
\end{equation}
Note that the original problem requires the ``physical'' variables $(u,v)$ to be real-valued; 
this corresponds to $(q,p)$ being real-valued too.
Also note that $\la \Lm^{\frac12} p, \Lm^{\frac12} p \ra = \la \Lm p, p\ra$.

\subsection{Diagonalization of the highest order: complex variables}
\label{subsec:diag highest}
To diagonalize the linear part $\pa_t q = \Lm p$, $\pa_t p = - \Lm q$
of system \eqref{1912.3}, we introduce complex variables. 

System \eqref{1912.3} and the Hamiltonian $H^{(1)}(q,p)$ in \eqref{def H(1)} 
are also meaningful, without any change, for \emph{complex} functions $q,p$. 
Thus we define the change of complex variables $(q,p) = \Phi^{(2)}(f,g)$ as
\begin{equation} \label{def Phi2}
(q,p) = \Phi^{(2)}(f,g) = \Big( \frac{f+g}{\sqrt2}, \frac{f-g}{i \sqrt2} \Big),
\qquad 
f = \frac{q + i p}{\sqrt2}, \quad  
g = \frac{q - i p}{\sqrt2}, 
\end{equation}
so that system \eqref{1912.3} becomes
\begin{equation} \label{syst uv}
\begin{cases}
\pa_t f = - i \Lm f - i \frac14 \la \Lm(f+g) , f+g \ra \Lm(f+g)
\\
\pa_t g = i \Lm g + i \frac14 \la \Lm(f+g) , f+g \ra \Lm(f+g)
\end{cases}
\end{equation}
where the pairing $\la \cdot , \cdot \ra$ denotes the integral of the product 
of any two complex functions
\begin{equation} \label{SPWC} 
\la w , h \ra := \int_{\T^d} w(x) h(x) \, dx 
= \sum_{j \in \Z^d \setminus \{ 0 \} } w_j h_{-j}, 
\quad w, h \in L^2(\T^d, \C).
\end{equation}
The map $\Phi^{(2)} : (f,g) \mapsto (q,p)$ in \eqref{def Phi2} 
is a $\C$-linear isomorphism of the space $L^2_0(\T^d,\C) \times L^2_0(\T^d,\C)$ 
of pairs of complex functions. 
When $(q,p)$ are real, $(f,g)$ are complex conjugate.
The restriction of $\Phi^{(2)}$ to the space
\begin{equation*} 
L^2_0(\T^d, c.c.) := 
\{ (f,g) \in L^2_0(\T^d,\C) \times L^2_0(\T^d,\C) : g = \overline{f} \}
\end{equation*}
of pairs of complex conjugate functions 
is an $\R$-linear isomorphism onto the space $L^2_0(\T^d,\R) \times L^2_0(\T^d,\R)$ 
of pairs of real functions. 
For $g = \overline{f}$, the second equation in \eqref{syst uv} is redundant, 
being the complex conjugate of the first equation. 
In other words, system \eqref{syst uv} has the following ``real structure'': 
it is of the form 
\begin{equation*} 
\pa_t \begin{pmatrix} f \\ g \end{pmatrix} 
= \mF(f,g) = \begin{pmatrix} \mF_1(f,g) \\ \mF_2(f,g) \end{pmatrix}
\end{equation*}
where the vector field $\mF(f,g)$ satisfies 
\begin{equation} \label{real vector field}
\mF_2(f, \overline{f}) = \overline{ \mF_1(f, \overline{f}) }.
\end{equation}
Under the transformation $\Phi^{(2)}$, the Hamiltonian system \eqref{1912.3} 
for complex variables $(q,p)$ becomes \eqref{syst uv}, which is the Hamiltonian system 
$\pa_t(f,g) = i J \gr H^{(2)}(f,g)$ with Hamiltonian 
$H^{(2)} = H^{(1)} \circ \Phi^{(2)}$, namely
\begin{equation*} 
H^{(2)}(f,g) = \la \Lm f, g \ra + \frac{1}{16} \la \Lm(f+g), f+g \ra^2,
\end{equation*} 
where $J$ is defined in \eqref{def H(1)}, 
$\la \cdot , \cdot \ra$ is defined in \eqref{SPWC},
and $\gr H^{(2)}$ is the gradient with respect to $\la \cdot , \cdot \ra$.
System \eqref{1912.3} for real $(q,p)$ (which corresponds to the original Kirchhoff equation)
becomes system \eqref{syst uv} restricted to the subspace $L^2_0(\T^d,c.c.)$ where 
$g = \overline{f}$. 

To complete the definition of the function spaces, 
for any real $s \geq 0$ we define 
\begin{equation*} 
H^s_0(\T^d,c.c.) := \{ (f,g) \in L^2_0(\T^d,c.c.) : f,g \in H^s_0(\T^d,\C) \}.
\end{equation*}

\section{Diagonalization of the order one} \label{sec:Bubaru}

Following a ``para-differential approach'', 
we note that the term $\la \Lm(f+g) , f+g \ra$ in \eqref{syst uv} 
plays the r\^ole of a coefficient, 
while $\Lm$ outside the scalar product is an operator of order one, 
in the sense that 
\[
\| \la \Lm f, g \ra \Lm h \|_s = \| h \|_{s+1} | \la \Lm f, g \ra |
\leq \| h \|_{s+1} \| f \|_{\frac12} \| g \|_{\frac12} 
\quad \forall s \geq 0, \ h \in H^{s+1}, \ f,g \in H^{\frac12}.
\]
Thus we write system \eqref{syst uv} as 
\begin{equation} \label{syst Q uv}
\pa_t \begin{pmatrix} f \\ g \end{pmatrix} 
= i 
\begin{pmatrix} - 1 - Q(f,g) & - Q(f,g) \\ 
Q(f,g) & 1 + Q(f,g) \end{pmatrix} 
\Lm 
\begin{pmatrix} f \\ g \end{pmatrix} 
\end{equation}
where
\begin{equation} \label{def Q(u,v)}
Q(f,g) := \frac14 \la \Lm(f+g) , f+g \ra.
\end{equation}
The aim of this section is to diagonalize system \eqref{syst Q uv} 
up to a bounded remainder, dealing with $Q(f,g)$ as a coefficient 
(even if it depends nonlinearly on the variables $(f,g)$).
On the real subspace $L^2_0(\T^d,c.c.)$ one has $g = \overline{f}$, and therefore
\[
Q(f,g) = \frac14 \la \Lm(f+g) , f+g \ra 
= \frac14 \la \Lm^\frac12 (f+ \overline{f}) , \Lm^\frac12 (f+ \overline{f}) \ra
= \int_{\T^d} \big( \Lm^\frac12 \Re(f) \big)^2 \, dx \geq 0,
\]
where $\Re(f)$ is the real part of $f$.  
Since $Q(f,g) \geq 0$, the matrix of the coefficients in \eqref{syst Q uv} 
has purely imaginary eigenvalues.  
For any $x \geq 0$, one has
\begin{equation} \label{2101.2}
\begin{pmatrix} - 1 - x & - x \\ 
x & 1 + x \end{pmatrix}
\begin{pmatrix} 1 & \rho(x) \\ \rho(x) & 1 \end{pmatrix} 
= 
\begin{pmatrix} 1 & \rho(x) \\ \rho(x) & 1 \end{pmatrix} 
\begin{pmatrix} - \sqrt{1+2x} & 0 \\ 
0 & \sqrt{1+2x} \end{pmatrix}
\end{equation}
where 
\begin{equation} \label{def rho}
\rho(x) := \frac{- x}{1 + x + \sqrt{1+2x}}\,.
\end{equation}
Note that $-1 < \rho(x) \leq 0$ for $x \geq 0$, 
so that the matrix 
$(\begin{smallmatrix} 1 & \rho(x) \\ \rho(x) & 1 \end{smallmatrix})$ is invertible. 
We define  
\begin{equation} \label{uv eta psi}
\begin{pmatrix} f \\ g \end{pmatrix} 
= \mM \begin{pmatrix} \eta \\ \psi \end{pmatrix},
\quad 
\mM = \mM(\rho) := \frac{1}{\sqrt{1-\rho^2}} \begin{pmatrix} 1 & \rho \\ 
\rho & 1 \end{pmatrix},
\end{equation}
where $\rho = \rho(Q(f,g))$, with $\rho$ defined in \eqref{def rho},  
and $Q(f,g)$ in \eqref{def Q(u,v)}.
The presence of the factor $(1 - \rho^2)^{-1/2}$ in the definition of $\mM$ 
is discussed in Remark \ref{rem:why factor} below.
To define a nonlinear change of variable expressing $(f,g)$ in terms of $(\eta,\psi)$ 
by using \eqref{uv eta psi}, we have to express the matrix $\mM$ as a function of $\eta,\psi$.  
Using \eqref{uv eta psi}, we calculate  
\[
Q(f,g) = \frac14 \la \Lm(f+g), f+g \ra
= \frac{1+\rho(Q(f,g))}{4(1 - \rho(Q(f,g)))} \la \Lm (\eta + \psi) , \eta + \psi \ra.
\]
From definition \eqref{def rho}, for any $x \geq 0$ one has 
\[
\frac{1-\rho(x)}{1+\rho(x)} = \sqrt{1+2x},
\]
whence
\begin{equation} \label{Q rad Q}
Q(f,g) \sqrt{1+2Q(f,g)} = \frac{1}{4} \la \Lm (\eta + \psi) , \eta + \psi \ra
= Q(\eta,\psi).
\end{equation}
The function $x \mapsto x \sqrt{1+2x}$ is invertible, and we denote by $\ph$ its inverse, 
\begin{equation} \label{def ph}
x \sqrt{1+2x} = y \quad \Leftrightarrow \quad x = \ph(y).
\end{equation}
Hence we can express $Q(f,g)$ in terms of $(\eta,\psi)$ as 
\begin{equation} \label{Q ph Q}
Q(f,g) = \ph \Big( \frac{1}{4} \la \Lm (\eta + \psi) , \eta + \psi \ra \Big)
= \ph(Q(\eta,\psi)) =: P(\eta,\psi).
\end{equation}
As a consequence, the matrix $\mM$ in \eqref{uv eta psi} 
can also be expressed as a function of $(\eta,\psi)$. 
In short, we denote it by $\mM(\eta,\psi)$, 
so that $\mM(\eta,\psi)$ is $\mM(\rho)$ where 
$\rho = \rho(\ph(Q(\eta,\psi))) = \rho(P(\eta,\psi))$, 
namely
\begin{equation} \label{def M(eta,psi)}
\mM(\eta,\psi) := \frac{1}{\sqrt{1-\rho^2(P(\eta,\psi))}} 
\begin{pmatrix} 1 & \rho(P(\eta,\psi)) \\ 
\rho(P(\eta,\psi)) & 1 \end{pmatrix}.
\end{equation}
We define the transformation $(f,g) = \Phi^{(3)}(\eta,\psi)$ by formula 
\eqref{uv eta psi} where $\mM = \mM(\eta,\psi)$.

\begin{lemma} \label{lemma:Phi3 inv}
Let $\Phi^{(3)}$ be the map 
\begin{equation} \label{def Phi3}
\Phi^{(3)}(\eta,\psi) = \mM(\eta,\psi) \begin{pmatrix} \eta \\ \psi \end{pmatrix}, 
\end{equation}
where $\mM(\eta,\psi)$ is defined in \eqref{def M(eta,psi)}, 
$\rho$ is defined in \eqref{def rho} 
and $P$ in \eqref{Q ph Q}. 
Then, for all real $s \geq \frac12$, 
the nonlinear map $\Phi^{(3)} : H^s_0(\T^d, c.c.) \to H^s_0(\T^d, c.c.)$ 
is invertible, continuous, with continuous inverse 
\begin{equation*} 
(\Phi^{(3)})^{-1} (f,g) = \frac{1}{\sqrt{1 - \rho^2(Q(f,g))}} 
\begin{pmatrix} 1 & - \rho(Q(f,g)) \\ - \rho(Q(f,g)) & 1 \end{pmatrix} 
\begin{pmatrix} f \\ g \end{pmatrix}. 
\end{equation*}
Moreover, for all $s \geq \frac12$, 
all $(\eta,\psi) \in H^s_0(\T^d,c.c.)$,  
one has 
\begin{equation*} 
\| \Phi^{(3)}(\eta,\psi) \|_s 
\leq C( \| \eta,\psi \|_{\frac12} ) \| \eta,\psi \|_s
\end{equation*}
for some increasing function $C$. 
The same estimate is satisfied by $(\Phi^{(3)})^{-1}$.
\end{lemma}

\begin{proof}
The regularity $H^{\frac12}$ guarantees that $Q(f,g)$ and $Q(\eta,\psi)$ are finite. 
The only point to prove is that $\Phi^{(3)}$ and its inverse map 
pairs of complex conjugate functions into pairs of complex conjugate functions.
Let $(\eta,\psi) \in H^{\frac12}_0(\T^d, c.c.)$. 
Then $Q(\eta,\psi)$, and therefore also $P(\eta,\psi) = \ph(Q(\eta,\psi))$, 
are real and $\geq 0$. Let $(f,g) = \Phi^{(3)}(\eta,\psi)$, namely 
\[
f = \frac{\eta + \rho \psi}{\sqrt{1 - \rho^2}}, \quad 
g = \frac{\rho \eta + \psi}{\sqrt{1 - \rho^2}}
\]
where $\rho = \rho(P(\eta,\psi))$. 
Since $\psi = \overline{\eta}$ and $\rho$ is real, 
we deduce that $\overline{f} = g$, 
and therefore $(f,g) \in H^{\frac12}_0(\T^d,c.c.)$. 
\end{proof}

Now we calculate how system \eqref{syst uv}, i.e.\  \eqref{syst Q uv}, 
transforms under the change of variable 
$(f,g) = \Phi^{(3)}(\eta, \psi) = \mM(\eta,\psi)[\eta,\psi]$.
We calculate
\[
\pa_t (f,g) 
= \pa_t \{ \mM(\eta,\psi)[ \eta, \psi] \}
= \mM(\eta, \psi) [\pa_t \eta, \pa_t \psi] 
+ \pa_t \{ \mM(\eta, \psi) \} [\eta, \psi],
\]
and 
\[
\pa_t \{ \mM(\eta, \psi) \}
= \frac{1}{(1 - \rho^2)^{3/2}} \begin{pmatrix} \rho & 1 \\ 1 & \rho \end{pmatrix} 
\pa_t \rho,
\]
\[
\pa_t \rho 
= \pa_t \{ \rho(\ph(Q(\eta,\psi))) \}
= \rho' \big( \ph(Q(\eta,\psi)) \big) \, 
\ph' \big( Q(\eta,\psi) \big) \, 
\frac12 \la \Lm(\eta + \psi), \pa_t \eta + \pa_t \psi \ra.
\]
By \eqref{2101.2} and \eqref{Q ph Q}, we have
\begin{align*}
& \begin{pmatrix} - i (1 + Q(f,g)) & - i Q(f,g) \\ 
i Q(f,g) & i (1 + Q(f,g)) \end{pmatrix}  
\begin{pmatrix} \Lm f \\ \Lm g \end{pmatrix}
= \mM(\eta,\psi) 
\begin{pmatrix} - 1 & 0 
\\ 0 & 1 \end{pmatrix} 
i \sqrt{1 + 2 P(\eta,\psi)} \begin{pmatrix} \Lm \eta \\ \Lm \psi \end{pmatrix}.
\end{align*}
Thus, applying $\mM(\eta,\psi)^{-1}$ from the left, \eqref{syst Q uv} becomes
\begin{equation} \label{syst intermedio 1}
\pa_t \begin{pmatrix} \eta \\ \psi \end{pmatrix} 
+ \mM(\eta,\psi)^{-1} \pa_t \{\mM(\eta,\psi)\} \begin{pmatrix} \eta \\ \psi \end{pmatrix} 
= \begin{pmatrix} - 1 & 0 \\ 0 & 1 \end{pmatrix} 
i \sqrt{1 + 2 P(\eta,\psi)} 
\begin{pmatrix} \Lm \eta \\ \Lm \psi \end{pmatrix}.
\end{equation}
We calculate
\[
\mM(\eta,\psi)^{-1} 
= \frac{1}{\sqrt{1 - \rho^2}} \begin{pmatrix} 1 & -\rho \\ -\rho & 1 \end{pmatrix},
\]
\[
\frac{1}{\sqrt{1 - \rho^2}} \begin{pmatrix} 1 & -\rho \\ -\rho & 1 \end{pmatrix}
\frac{1}{(1 - \rho^2)^{\frac32}} \begin{pmatrix} \rho & 1 \\ 1 & \rho \end{pmatrix}
= \frac{1}{1 - \rho^2} \begin{pmatrix} 0 & 1 \\ 1 & 0 \end{pmatrix},
\]
\[
\rho'(x) 
= \frac{-1}{\,\sqrt{1 + 2x} \, (1 + x + \sqrt{1+2x})\,},
\]
\[
\frac{1}{1 - \rho^2(x)} \cdot \frac{-1}{\,\sqrt{1 + 2x} \, (1 + x + \sqrt{1+2x})\,}
= \frac{-1}{2(1 + 2x)},
\]
\[
\ph'(y) = \frac{\sqrt{1+2\ph(y)}}{1+3 \ph(y)}. 
\]
Hence
\[
\mM(\eta,\psi)^{-1} \pa_t \{ \mM(\eta,\psi) \} \begin{pmatrix} \eta \\ \psi \end{pmatrix}
= \mK(\eta,\psi) \begin{pmatrix} \pa_t \eta \\ \pa_t \psi \end{pmatrix}
\]
where $\mK(\eta,\psi)$ is the operator 
\begin{equation*} 
\mK(\eta,\psi) \begin{pmatrix} \a \\ \b \end{pmatrix} 
:= \begin{pmatrix} \psi \\ \eta \end{pmatrix} 
F(\eta,\psi) \la \Lm(\eta + \psi), \a + \b \ra
\end{equation*}
and $F(\eta,\psi)$ is the scalar factor
\begin{equation} \label{def F factor}
F(\eta,\psi) 
:= \frac{- 1}{4 (1 + 3 P(\eta,\psi)) \sqrt{1 + 2 P(\eta,\psi)}}. 
\end{equation}
By induction, for all $n = 1,2,3,\ldots$ one has
\[
\mK^n(\eta,\psi) \begin{pmatrix} \a \\ \b \end{pmatrix} 
= \begin{pmatrix} \psi \\ \eta \end{pmatrix} F(\eta,\psi)^n \la \Lm(\eta + \psi), \eta + \psi \ra^{n-1} \la \Lm(\eta + \psi), \a + \b \ra.
\]
Thus, by geometric series,  
\[
\sum_{n=1}^\infty (- \mK(\eta,\psi))^n \begin{pmatrix} \a \\ \b \end{pmatrix} 
= \begin{pmatrix} \psi \\ \eta \end{pmatrix} \la \Lm(\eta + \psi), \a + \b \ra
\frac{-F(\eta,\psi)}{1 + F(\eta,\psi) \la \Lm(\eta + \psi), \eta + \psi \ra}
\]
provided that $|F(\eta,\psi) \la \Lm(\eta + \psi), \eta + \psi \ra| < 1$. 
Since $\la \Lm(\eta + \psi), \eta + \psi \ra = 4 Q(\eta,\psi)$, using \eqref{def F factor}, 
\eqref{Q ph Q}, \eqref{Q rad Q} and \eqref{def ph}, 
we have 
\begin{align*}
F(\eta,\psi) \la \Lm(\eta + \psi), \eta + \psi \ra
& = \frac{- Q(\eta,\psi)}{(1 + 3 P(\eta,\psi)) \sqrt{1 + 2 P(\eta,\psi)}}
= \frac{- Q(f,g) }{1 + 3 Q(f,g)}, 
\end{align*}
whence $|F(\eta,\psi) \la \Lm(\eta + \psi), \eta + \psi \ra| < 1/3$ 
for all $Q(f,g) \geq 0$, and the geometric series converges. 
Using the same identities, 
we also obtain that 
\[
\frac{-F(\eta,\psi)}{1 + F(\eta,\psi) \la \Lm(\eta + \psi), \eta + \psi \ra}
= \frac{1}{4 (1+2 P(\eta,\psi))^{\frac32}}.
\]
Hence 
\[
( I + \mK(\eta,\psi))^{-1} \begin{pmatrix} \a \\ \b \end{pmatrix}
= \begin{pmatrix} \a \\ \b \end{pmatrix} 
+ \begin{pmatrix} \psi \\ \eta \end{pmatrix} \la \Lm(\eta + \psi), \a + \b \ra
\frac{1}{4 (1+2 P(\eta,\psi))^{\frac32}}.
\]
Then system \eqref{syst intermedio 1} becomes 
\[
\pa_t \begin{pmatrix} \eta \\ \psi \end{pmatrix}
= ( I + \mK(\eta,\psi))^{-1} \begin{pmatrix} - i \sqrt{1 + 2 P(\eta,\psi)} & 0 \\ 
0 & i \sqrt{1 + 2 P(\eta,\psi)} \end{pmatrix} 
\begin{pmatrix} \Lm \eta \\ \Lm \psi \end{pmatrix},
\]
which is 
\[
\pa_t \begin{pmatrix} \eta \\ \psi \end{pmatrix}
= \begin{pmatrix} - i \sqrt{1 + 2 P(\eta,\psi)} \, \Lm \eta \\ 
i \sqrt{1 + 2 P(\eta,\psi)} \, \Lm \psi \end{pmatrix}
+ \begin{pmatrix} \psi \\ \eta \end{pmatrix} \la \Lm(\eta + \psi), \Lm(\psi - \eta) \ra
\frac{i}{4(1+2 P(\eta,\psi))},
\]
namely
\begin{equation} \label{syst 6 dic}
\begin{cases}
\pa_t \eta = - i \sqrt{1 + 2 P(\eta,\psi)} \, \Lm \eta
+ \dfrac{i}{4(1+2 P(\eta,\psi))} 
\Big( \la \Lm \psi, \Lm \psi \ra - \la \Lm \eta , \Lm \eta \ra \Big) \psi
\\
\pa_t \psi = i \sqrt{1 + 2 P(\eta,\psi)} \, \Lm \psi
+ \dfrac{i}{4(1+2 P(\eta,\psi))} 
\Big( \la \Lm \psi, \Lm \psi \ra - \la \Lm \eta , \Lm \eta \ra \Big) \eta.
\end{cases}
\end{equation}
We remark that system \eqref{syst 6 dic} is diagonal at the order one, 
i.e.\ the coupling of $\eta$ and $\psi$ (except for the coefficients) 
is confined to terms of order zero. 
Also note that the coefficients of \eqref{syst 6 dic} are finite for $\eta,\psi \in H^1_0$, while the coefficients in \eqref{syst uv} are finite for $f,g \in H^{\frac12}_0$:  
the regularity threshold of the transformed system is $\frac12$ higher than before. 

We also note that the real structure is preserved, 
namely the second equation in \eqref{syst 6 dic} is the complex conjugate of the first one, 
or, in other words, the vector field in \eqref{syst 6 dic} satisfies property 
\eqref{real vector field}.

\begin{remark} \label{rem:why factor}
It would be tempting to use the transformation 
\begin{equation} \label{tempting}
\begin{pmatrix} f \\ g \end{pmatrix} 
= \frac{1}{1 + \rho} \begin{pmatrix} 1 & \rho \\ \rho & 1 \end{pmatrix}
\begin{pmatrix} \eta \\ \psi \end{pmatrix} 
\end{equation}
instead of \eqref{uv eta psi}, because \eqref{tempting} preserves the formula of $Q$: 
\[
Q(f,g) = \frac14 \la \Lm(f+g), f+g \ra 
= \frac14 \la \Lm(\eta+\psi), \eta + \psi \ra = Q(\eta,\psi),
\]
avoiding the use of the inverse function $\ph$.
However, using \eqref{tempting} 
would produce a diagonal term of order zero in the transformed system
which does not cancel out in the energy estimate
(in fact, on the real subspace $\psi = \bar\eta$ such a diagonal term 
has a real coefficient). 
The factor $(1 - \rho^2)^{-\frac12}$ in \eqref{uv eta psi} 
is the only (up to constant factors) choice that eliminates 
those diagonal terms of order zero.
This property is related to the symplectic structure of \eqref{syst uv}.
\end{remark}

\section{Normal form transformation} 
\label{sec:NF} 

Let $(\eta,\psi)$ be a solution of \eqref{syst 6 dic}, 
with $\psi = \overline{\eta}$. 
Then its a priori energy estimate is
\begin{align*}
\pa_t (\| \eta \|_s^2) 
& = \pa_t \la \Lm^s \eta, \Lm^s \psi \ra 
= \la \pa_t \eta , \Lm^{2s} \psi \ra 
+ \la \Lm^{2s} \eta, \pa_t \psi \ra
\\ 
& = \la - i \sqrt{1 + 2 P(\eta,\psi)} \, \Lm \eta
+ \dfrac{i ( \la \Lm \psi, \Lm \psi \ra - \la \Lm \eta , \Lm \eta \ra )}{4(1+2 P(\eta,\psi))} \psi , \Lm^{2s} \psi \ra 
\\ 
& \quad \ 
+ \la \Lm^{2s} \eta, i \sqrt{1 + 2 P(\eta,\psi)} \, \Lm \psi
+ \dfrac{i ( \la \Lm \psi, \Lm \psi \ra - \la \Lm \eta , \Lm \eta \ra )}{4(1+2 P(\eta,\psi))} \eta \ra
\\ 
& = \dfrac{i ( \la \Lm \psi, \Lm \psi \ra - \la \Lm \eta , \Lm \eta \ra )}{4(1+2 P(\eta,\psi))} 
( \la \psi , \Lm^{2s} \psi \ra + \la \Lm^{2s} \eta, \eta \ra )
\\ 
& \leq \| \eta \|_1^2 \| \eta \|_s^2. 
\end{align*}
This gives the local existence in $H^1_0$ in a time interval $[0,T]$ 
with $T = O(\| \eta(0) \|_1^{-2})$. 
We note that the terms $(- i \sqrt{1 + 2 P(\eta,\psi)} \, \Lm \eta, 
i \sqrt{1 + 2 P(\eta,\psi)} \, \Lm \psi)$ 
give no contribution to the energy estimate, 
thanks to their diagonal structure, which was obtained in the previous section.  
Hence, to improve the energy estimate and to extend the existence time, 
there is no need to modify those terms. 
In fact, reparametrizing the time variable, the coefficient $\sqrt{1 + 2 P(\eta,\psi)}$
could be normalized to 1; however, as just noticed, this is not needed to 
our purposes. 

The next step in our proof is the cancellation of the cubic terms contributing 
to the energy estimate. 
We write \eqref{syst 6 dic} as
\begin{equation} \label{syst DBR}
\pa_t (\eta,\psi) = X(\eta,\psi) = \mD_1(\eta,\psi) + \mD_{\geq 3}(\eta,\psi) 
+ \mB_3(\eta,\psi) + \mR_{\geq 5}(\eta,\psi)
\end{equation}
where $\mD_1(\eta,\psi)$ is the linear component of the unbounded diagonal operator
$\mD(\eta,\psi) = i \sqrt{1 + 2 P(\eta,\psi)} (- \Lm \eta, \Lm \psi)$,
namely
\begin{equation*} 
\mD_1(\eta,\psi) = \begin{pmatrix} - i \Lm \eta \\ i \Lm \psi \end{pmatrix},
\end{equation*} 
$\mD_{\geq 3}(\eta,\psi)$ is the difference $\mD - \mD_1$, 
namely 
\begin{equation} \label{def mD geq 3}
\mD_{\geq 3}(\eta,\psi) = \Big( \sqrt{1 + 2 P(\eta,\psi)} \, - 1 \Big) 
\begin{pmatrix} - i \Lm \eta \\ i \Lm \psi \end{pmatrix},
\end{equation} 
$\mB_3(\eta,\psi)$ is the cubic component of the bounded, off-diagonal term
\begin{equation} \label{def mB}
\mB_3(\eta,\psi) = 
\frac{i}{4} 
\Big( \la \Lm \psi, \Lm \psi \ra - \la \Lm \eta , \Lm \eta \ra \Big) 
\begin{pmatrix} \psi \\ \eta \end{pmatrix}
\end{equation}
and $\mR_{\geq 5}(\eta,\psi)$ is the bounded remainder of higher homogeneity degree
\begin{equation} \label{def mR geq 5}
\mR_{\geq 5}(\eta,\psi) = 
\frac{- i P(\eta,\psi)}{2 (1 + 2 P(\eta,\psi))} 
\Big( \la \Lm \psi, \Lm \psi \ra - \la \Lm \eta , \Lm \eta \ra \Big) 
\begin{pmatrix} \psi \\ \eta \end{pmatrix}.
\end{equation}
The aim of this section is to remove $\mB_3$ 
($\mD$ gives no contribution to the energy estimate, 
and $\mR_{\geq 5}(\eta,\psi) = O((\eta,\psi)^5)$ gives a contribution of higher order). 

We consider a transformation $(\eta,\psi) = \Phi^{(4)}(w,z)$ of the form 
\begin{equation} \label{def Phi4}
\begin{pmatrix} \eta \\ \psi \end{pmatrix} = \Phi^{(4)}(w,z)
= (I + M(w,z) ) \begin{pmatrix} w \\ z \end{pmatrix}, 
\end{equation}
\[ 
M(w,z) = \begin{pmatrix} M_{11}(w,z) & M_{12}(w,z) \\ 
M_{21}(w,z) & M_{22}(w,z) \end{pmatrix},
\]
\[
M_{ij}(w,z) = A_{ij}[w,w] + B_{ij}[w,z] + C_{ij}[z,z],
\quad i,j \in \{ 1,2 \},
\]
where $A_{ij}$, $B_{ij}$, $C_{ij}$ are bilinear maps. 
We also denote 
\[
A[w,w] = \begin{pmatrix} A_{11}[w,w] & A_{12}[w,w] \\ 
A_{21}[w,w] & A_{22}[w,w] \end{pmatrix}
\]
and similarly for $B[w,z]$ and $C[z,z]$. 
We assume that 
\[
A[w_1, w_2] = A[w_2, w_1], \quad 
C[z_1, z_2] = C[z_2, z_1] \quad \forall w_1, w_2, z_1, z_2.
\]
We calculate how system \eqref{syst 6 dic} transforms under 
the change of variable $(\eta,\psi) = \Phi^{(4)}(w,z)$. 
One has 
\begin{align*}
\pa_t \begin{pmatrix} \eta \\ \psi \end{pmatrix} 
= (I + M(w,z)) \begin{pmatrix} \pa_t w \\ \pa_t z \end{pmatrix}
+ \{ \pa_t M(w,z) \} \begin{pmatrix} w \\ z \end{pmatrix} 
\end{align*}
and
\begin{align*}
\pa_t M(w,z) 
& = \pa_t ( A[w,w] + B[w,z] + C[z,z] )
\\ & 
= 2 A[w, \pa_t w] + B[\pa_t w, z] + B[w, \pa_t z] + 2 C[z, \pa_t z].
\end{align*}
Thus 
\[
\pa_t \begin{pmatrix} \eta \\ \psi \end{pmatrix} 
= (I + K(w,z)) \begin{pmatrix} \pa_t w \\ \pa_t z \end{pmatrix},
\]
where 
\begin{equation} \label{def K(f,g)}
K(w,z)\begin{pmatrix} \a \\ \b \end{pmatrix}
= M(w,z) \begin{pmatrix} \a \\ \b \end{pmatrix}
+ \{ 2 A[w, \a] + B[\a, z] + B[w, \b] + 2 C[z, \b] \} 
\begin{pmatrix} w \\ z \end{pmatrix}.
\end{equation}
System \eqref{syst 6 dic}, namely \eqref{syst DBR}, becomes 
\begin{equation} \label{systo extra fg}
(I + K(w,z)) \begin{pmatrix} \pa_t w \\ \pa_t z \end{pmatrix}
= X (\Phi^{(4)}(w,z)). 
\end{equation}
Assume that, by Neumann series, $I + K(w,z)$ is invertible 
(this will be proved below, after the choice of $M(w,z)$). 
Thus \eqref{systo extra fg} becomes
\begin{equation} \label{def X+}
\pa_t \begin{pmatrix} w \\ z \end{pmatrix}
= (I + K(w,z))^{-1} X(\Phi^{(4)}(w,z))
=: X^+(w,z).
\end{equation}
Since $X = \mD_1 + \mD_{\geq 3} + \mB_3 + \mR_{\geq 5}$ 
and $(I + K(w,z))^{-1} = I - K(w,z) + \tilde K(w,z)$, 
where $\tilde K(w,z) := \sum_{n=2}^\infty (- K(w,z))^n$, 
we calculate  
\begin{align}
X^+(w,z) 
& = \mD_1(w,z) 
+ \mD_1 \Big( M(w,z) \begin{pmatrix} w \\ z \end{pmatrix} \Big)
- K(w,z) \mD_1 (w,z) 
\notag \\ & \quad \ 
- K(w,z) \mD_1 \Big( M(w,z) \begin{pmatrix} w \\ z \end{pmatrix} \Big)
+ \tilde K(w,z) \mD_1( \Phi^{(4)}(w,z))
+ \mB_3(w,z) 
\notag \\ & \quad \ 
+ (I + K(w,z))^{-1} \mD_{\geq 3} ( \Phi^{(4)}(w,z))
+ (I + K(w,z))^{-1} \mR_{\geq 5} ( \Phi^{(4)}(w,z))
\notag \\ & \quad \ 
+ [ \mB_3 ( \Phi^{(4)}(w,z)) - \mB_3(w,z) ]
+ \big( - K(w,z) + \tilde K(w,z) \big) \mB_3 ( \Phi^{(4)}(w,z)).
\label{X+}
\end{align}
We look for $M(w,z)$ such that the cubic terms
\begin{equation} \label{homological}
X_3^+ (w,z) 
:= \mD_1 \Big( M(w,z) \begin{pmatrix} w \\ z \end{pmatrix} \Big)
- K(w,z) \mD_1 (w,z) 
+ \mB_3(w,z) 
\end{equation}
give no contribution to the energy estimate. 
Note that $X_3^+$ is not the entirety of the cubic terms of $X^+$,  
because a cubic term also arises from 
$(I + K(w,z))^{-1} \mD_{\geq 3} ( \Phi^{(4)}(w,z))$; 
however, this cubic term is diagonal, 
it does not contribute to the energy estimate, 
and it does not interact with the off-diagonal cubic term $\mB_3(w,z)$, 
therefore we do not include it in \eqref{homological}.  

The first component $(X_3^+)_1(w,z)$ 
of the vector $X_3^+(w,z)$ in \eqref{homological} is
\begin{align*}
(X_3^+)_1(w,z) 
& = - i \Lm M_{11}(w,z) w - i \Lm M_{12}(w,z) z 
+ i M_{11}(w,z) \Lm w - i M_{12}(w,z) \Lm z
\\ & \quad \ 
- \{ 2 A_{11}[w, -i \Lm w] + B_{11}[-i\Lm w, z] + B_{11}[w, i \Lm z] + 2 C_{11}[z, i\Lm z] \} w
\\ & \quad \ 
- \{ 2 A_{12}[w, -i \Lm w] + B_{12}[-i\Lm w, z] + B_{12}[w, i \Lm z] + 2 C_{12}[z, i\Lm z] \} z
\\ & \quad \ 
+ \frac{i}{4} \Big( \la \Lm z, \Lm z \ra - \la \Lm w, \Lm w \ra \Big) z. 
\end{align*}
We choose 
\[
M_{11} = 0, \quad B_{12} = 0,
\]
because $M_{11}$ is not involved in the calculation to remove 
the off-diagonal terms (those ending with $z$), 
and there are no terms of the form [coefficient $O(wz)$ times $z$] to remove. 
It remains 
\begin{align*}
(X_3^+)_1(w,z) 
& = - i \Lm A_{12}[w,w] z - i \Lm C_{12}[z,z] z
- i A_{12}[w,w] \Lm z - i C_{12}[z,z] \Lm z
\\ & \quad \ 
+ 2 i A_{12}[w, \Lm w] z - 2 i C_{12}[z, \Lm z] z
+ \frac{i}{4} \big( \la \Lm z, \Lm z \ra - \la \Lm w, \Lm w \ra \big) z.
\end{align*}
We look for $A_{12}, C_{12}$ of the form
\begin{align*}
A_{12}[u , v] h 
& = \sum_{j,k \in \Z^d \setminus \{ 0 \}} u_j v_{-j} a_{12}(j,k) h_k e^{ik \cdot x} 
\quad \forall u,v,h,  
\\
C_{12}[u, v] h 
& = \sum_{j,k \in \Z^d \setminus \{ 0 \}} u_j v_{-j} c_{12}(j,k) h_k e^{ik \cdot x}
\quad \forall u,v,h,  
\end{align*}
for some coefficients $a_{12}(j,k), c_{12}(j,k)$ to be determined, 
where $u_j, v_j, h_k$ are the Fourier coefficients of any functions $u(x), v(x), h(x)$.
Hence 
\begin{align*}
(X_3^+)_1 (w,z)
& = \sum_{j,k \neq 0} w_j w_{-j} z_k e^{ik \cdot x} 
\Big( 2i (|j| - |k|) a_{12}(j,k) - \frac{i}{4} |j|^2 \Big)
\\ & \quad \ 
+ \sum_{j,k \neq 0} z_j z_{-j} z_k e^{ik \cdot x} 
\Big( -2i (|j| + |k|) c_{12}(j,k) + \frac{i}{4} |j|^2 \Big).
\end{align*}
We fix
\begin{equation} \label{fix a12 c12}
a_{12}(j,k) := \begin{cases} \dfrac{|j|^2}{8(|j| - |k|)} 
& \text{if} \ |j| \neq |k|, 
\\
0 & \text{if} \ |j| = |k|,
\end{cases}
\quad \qquad
c_{12}(j,k) := \frac{|j|^2}{8(|j| + |k|)}.
\end{equation}
Thus the operators $A_{12}, C_{12}$ are
\begin{align}
\label{fix A12}
A_{12}[u, v] h
& = \sum_{j,k \neq 0, \, |j| \neq |k|} 
u_j v_{-j} \frac{|j|^2}{8(|j| - |k|)} h_k e^{ik \cdot x},    
\\
C_{12}[u, v] h 
& = \sum_{j,k \neq 0} u_j v_{-j} \frac{|j|^2}{8(|j| + |k|)} h_k e^{ik \cdot x},
\label{fix C12}
\end{align}
and 
\begin{equation} \label{X3+ 1}
(X_3^+)_1(w,z) 
= - \frac{i}{4} \sum_{j,k \neq 0,\, |k| = |j|} w_j w_{-j} |j|^2 z_k e^{ik \cdot x}.  
\end{equation}
The analogous calculation for the second component $(X_3^+)_2(w,z)$ of the vector 
in \eqref{homological} leads to the choice 
\begin{equation} \label{fix A21 C21}
M_{22} = 0, \quad 
B_{21} = 0, \quad 
A_{21} = C_{12}, \quad 
C_{21} = A_{12},
\end{equation}
and it remains 
\begin{equation} \label{X3+ 2}
(X_3^+)_2(w,z) 
= \frac{i}{4} \sum_{j,k \neq 0,\, |k| = |j|} z_j z_{-j} |j|^2 w_k e^{ik \cdot x}.  
\end{equation}
We will see below (see \eqref{below is here}) 
that the remaining cubic terms $(X_3^+)_1(w,z)$ and $(X_3^+)_2(w,z)$ 
do not contribute to the growth of the Sobolev norms in the energy estimate. 

Now that $M$ has been fixed, 
we have to prove the invertibility of $(I + K(w,z))$ by Neumann series. 
Since 
\begin{equation} \label{formula M}
M(w,z) = \begin{pmatrix} 0 & A_{12}[w,w] + C_{12}[z,z] \\ 
A_{21}[w,w] + C_{21}[z,z] & 0 \end{pmatrix},
\end{equation}
recalling \eqref{def K(f,g)} one has
\begin{equation} \label{formula K}
K(w,z) \begin{pmatrix} \a \\ \b \end{pmatrix} 
= \begin{pmatrix} 
A_{12}[w,w] \b + C_{12}[z,z] \b + 2 A_{12}[w,\a] z + 2 C_{12}[z,\b] z
\\
A_{21}[w,w] \a + C_{21}[z,z] \a + 2 A_{21}[w,\a] w + 2 C_{21}[z,\b] w 	
\end{pmatrix},
\end{equation}
namely
\begin{equation*} 
K(w,z) \begin{pmatrix} \a \\ \b \end{pmatrix} 
= M(w,z) \begin{pmatrix} \a \\ \b \end{pmatrix} 
+ E(w,z) \begin{pmatrix} \a \\ \b \end{pmatrix} 
\end{equation*}
where $M(w,z)$ is given in \eqref{formula M} and
\begin{equation*} 
E(w,z) \begin{pmatrix} \a \\ \b \end{pmatrix}  
:= \begin{pmatrix} 2 A_{12}[w,\a] z + 2 C_{12}[z,\b] z \\
2 A_{21}[w,\a] w + 2 C_{21}[z,\b] w \end{pmatrix}.
\end{equation*}
To estimate matrix operators and vectors in $H^s_0(\T^d,c.c.)$, 
we define $\| (w,z) \|_s := \| w \|_s = \| z \|_s$ 
for every pair $(w,z) = (w, \overline{w})$ of complex conjugate functions.  

\begin{lemma}
Let $A_{12}, C_{12}$ be the operators defined in \eqref{fix A12}, 
\eqref{fix C12}, and let $m_0$ be defined in \eqref{def m0}.
For all complex functions $u,v,h$, all real $s \geq 0$, 
\begin{equation} \label{stima A12 basic}
\| A_{12}[u, v] h \|_s 
\leq \frac38 \| u \|_{m_0} \| v \|_{m_0} \| h \|_s,
\quad 
\| C_{12}[u, v] h \|_s 
\leq \frac{1}{16} \| u \|_1 \| v \|_1 \| h \|_s.
\end{equation}
\end{lemma}

\begin{proof}
In dimension $d = 1$, one has $| |j| - |k| | \geq 1$ for $|j| \neq |k|$. 
Therefore, by H\"older's inequality, 
\begin{align*}
\| A_{12}[u, v] h \|_s^2 
& = \sum_{k \neq 0} \Big| \sum_{j \neq 0, \, |j| \neq |k|}  
u_j v_{-j} \frac{|j|^2}{8(|j| - |k|)} h_k \Big|^2 |k|^{2s} 
\\ & 
\leq     
\frac{1}{64} \sum_{k \neq 0} \Big( \sum_{j \neq 0, \, |j| \neq |k|}  
|u_j| |j| |v_{-j}| |j| \Big)^2 |h_k|^2 |k|^{2s} 
\leq \frac{1}{64} \| u \|_1^2 \| v \|_1^2 \| h \|_s^2.
\end{align*}
In dimension $d \geq 2$, we observe that
\begin{equation} \label{3j}
\frac{1}{||j| - |k||} \leq 3 |j| 
\quad \forall j,k \in \Z^d \setminus \{ 0 \}, \ |j| \neq |k|.
\end{equation}
If $| |j| - |k| | \geq 1$, then \eqref{3j} holds because $|j| \geq 1$.
Let $| |j| - |k| | < 1$, with $|j| \neq |k|$. 
Then $|k| < |j|+1$, and, 
since $(|j| - |k|)(|j| + |k|) = |j|^2 - |k|^2$ is a nonzero integer, one has 
\[
\frac{1}{||j| - |k||} = \frac{|j| + |k|}{||j|^2 - |k|^2|} 
\leq |j| + |k| < 2 |j| + 1 \leq 3|j|.
\]
Hence 
\begin{align*}
\| A_{12}[u, v] h \|_s^2 
& = \sum_{k \neq 0} \Big| \sum_{j \neq 0, \, |j| \neq |k|}  
u_j v_{-j} \frac{|j|^2}{8(|j| - |k|)} h_k \Big|^2 |k|^{2s} 
\\ & 
\leq     
\sum_{k \neq 0} \Big( \sum_{j \neq 0, \, |j| \neq |k|}  
|u_j| |v_{-j}| \frac38 |j|^3 \Big)^2 |h_k|^2 |k|^{2s} 
\leq \frac{9}{64} \| u \|_{\frac32}^2 \| v \|_{\frac32}^2 \| h \|_s^2.
\end{align*}
To estimate $C_{12}$, we use the bound $8(|j| + |k|) \geq 16$, 
which holds in any dimension.  
\end{proof}

\begin{lemma} 
For all $s \geq 0$, all $(w,z) \in H^{m_0}_0(\T^d, c.c.)$, 
$(\a,\b) \in H^s_0(\T^d,c.c.)$ one has
\begin{align} \label{stima M}
\Big\| M(w,z) \begin{pmatrix} \a \\ \b \end{pmatrix} \Big\|_s 
& \leq \frac{7}{16} \| w \|_{m_0}^2 \|\a \|_s ,
\\
\Big\| K(w,z) \begin{pmatrix} \a \\ \b \end{pmatrix} \Big\|_s 
& \leq \frac{7}{16} \| w \|_{m_0}^2 \|\a \|_s 
+ \frac78 \| w \|_{m_0} \| w \|_s \| \a \|_{m_0} ,
\label{stima K}
\end{align}
where $m_0$ is defined in \eqref{def m0}. 
For $\| w \|_{m_0} < \frac12$, 
the operator $(I + K(w,z)) : H^{m_0}_0(\T^d, c.c.)$ $\to H^{m_0}_0(\T^d, c.c.)$ 
is invertible, with inverse 
\[
(I + K(w,z))^{-1} = I - K(w,z) + \tilde K(w,z), \quad 
\tilde K(w,z) := \sum_{n=2}^\infty (- K(w,z))^n,
\]
satisfying
\begin{equation*} 
\Big\| (I + K(w,z))^{-1} \begin{pmatrix} \a \\ \b \end{pmatrix} \Big\|_s 
\leq C (\| \a \|_s + \| w \|_{m_0} \| w \|_s \| \a \|_{m_0}),
\end{equation*}
for all $s \geq 0$, where $C$ is a universal constant.
\end{lemma}

\begin{proof}
Use \eqref{formula M}, \eqref{formula K}, \eqref{stima A12 basic} and Neumann series.
\end{proof}

By contraction lemma, we prove that the nonlinear, continuous map $\Phi^{(4)}$ 
is invertible in a ball around the origin. 

\begin{lemma} \label{lemma:Phi4 inv}
For all $(\eta, \psi) \in H^{m_0}_0(\T^d, c.c.)$ 
in the ball $\| \eta \|_{m_0} \leq \frac14$, 
there exists a unique $(w,z) \in H^{m_0}_0(\T^d, c.c.)$ such that 
$\Phi^{(4)}(w,z) = (\eta,\psi)$, with $\| w \|_{m_0} \leq 2 \| \eta \|_{m_0}$. 
If, in addition, $\eta \in H^s_0$ for some $s > m_0$, then $w$ also belongs to $H^s_0$, 
and $\| w \|_s \leq 2 \| \eta \|_s$.
This defines the continuous inverse map $(\Phi^{(4)})^{-1} : H^s_0(\T^d, c.c.) \cap 
\{ \| \eta \|_{m_0} \leq \frac14 \}$ $\to H^s_0(\T^d, c.c.)$.
\end{lemma}

\begin{proof} 
\emph{Existence}.
Given $(\eta,\psi)$, the problem of finding $(w,z)$ such that $\Phi^{(4)}(w,z) = (\eta,\psi)$ 
is the fixed point problem $\Psi(w,z) = (w,z)$, where 
\[
\Psi(w,z) := \begin{pmatrix} \eta \\ \psi \end{pmatrix} 
- M(w,z) \begin{pmatrix} w \\ z \end{pmatrix}.
\]
Let $B_R := \{ (w,z) \in H^{m_0}_0(\T^d,c.c.) : \| w \|_{m_0} \leq R \}$. 
By \eqref{stima M}, $\Psi$ maps $B_R \to B_R$ if $\| \eta \|_{m_0} + \frac{7}{16} R^3 \leq R$. 
Since 
\begin{align}
& M(w_1, z_1) \begin{pmatrix} w_1 \\ z_1 \end{pmatrix}
- M(w_2, z_2) \begin{pmatrix} w_2 \\ z_2 \end{pmatrix}
\notag \\ & \quad 
= \int_0^1 K \big( w_2 + \th (w_1 - w_2) ,  z_2 + \th (z_1 - z_2) \big) \, d\th 
\begin{pmatrix} w_1 - w_2 \\ z_1 - z_2 \end{pmatrix},
\label{Lip}
\end{align}
by \eqref{stima K} $\Psi$ is a contraction if $\frac{21}{16} R^2 < 1$. 
We choose $R = 2 \| \eta \|_{m_0}$, so that $\Psi$ is a contraction in $B_R$ 
if $\| \eta \|_{m_0} \leq \frac14$. 
As a consequence, there exists a unique fixed point $(w,z) = \Psi(w,z)$ 
in $B_R$, with $\| w \|_{m_0} \leq R = 2 \| \eta \|_{m_0}$. 

\emph{Regularity}. Assume, in addition, that $\eta \in H^s$. 
The fixed point $w$ is the limit in $H^{m_0}$ of the sequence $w_n := \Psi(w_{n-1})$, $w_0 := 0$. 
We write $w$ as the sum of the telescoping series $\sum_{n=0}^\infty h_n$, 
which converges in $H^{m_0}$, where $h_n := w_{n+1} - w_n$. 
Since $\eta \in H^s$ and $\Psi$ maps $H^s \to H^s$, 
then $w_n \in H^s$ for all $n$. 
By \eqref{Lip}, 
\begin{equation} \label{hn low}
\| h_n \|_{m_0} \leq B^n \| h_0 \|_{m_0} \quad \forall n \geq 0, 
\end{equation}
where $B := \frac{21}{16} R^2$. Note that $h_0 = w_1 = \eta$. 
By induction, we prove that 
\begin{equation} \label{i ii}
(i) \  
\| w_n \|_s \leq \rho_s; 
\qquad 
(ii) \  
\| h_n \|_s \leq B^n \| h_0 \|_s + n B^{n-1} A_s \| h_0 \|_{m_0}
\end{equation}
for some constants $\rho_s, A_s$ to determine. 

At $n=0$ \eqref{i ii} trivially holds. 
At $n=1$, $(i)$ holds if $\rho_s \geq \| \eta \|_s$, 
and $(ii)$ holds because, by \eqref{stima M},
$\| h_1 \|_s = \| M(\eta,\psi) \binom{\eta}{\psi} \|_s 
\leq \frac{7}{16} \| \eta \|_{m_0}^2 \| \eta \|_s$ 
and $h_0 = \eta$. 

Assume that \eqref{i ii} holds for all $k \leq n$, for some $n \geq 1$. 
Using \eqref{Lip}, \eqref{stima K}, $(i)_n$ and $(i)_{n-1}$,
we deduce that 
$\| h_{n+1} \|_s 
\leq \frac{7}{16} R^2 \| h_n \|_s + \frac78 R \rho_s \| h_n \|_{m_0}$. 
Using $(ii)_n$ and \eqref{hn low}, this is 
$\leq (\frac{7}{16} R^2 B^n) \| h_0 \|_s 
+ (\frac{7}{16} R^2 n B^{n-1} A_s + \frac78 R \rho_s B^n ) \| h_0 \|_{m_0}$.
Since $B = \frac{21}{16} R^2$, $(ii)_{n+1}$ holds provided that $\frac78 R \rho_s \leq A_s$. 
We fix $A_s = \frac78 R \rho_s$. 

To prove $(i)_{n+1}$, we use $(ii)_k$ for $k=0,\ldots,n$, 
and we estimate 
$
\| w_{n+1} \|_s \leq \sum_{k=0}^n \| h_k \|_s 
\leq \sum_{k=0}^n B^k \| h_0 \|_s + \sum_{k=0}^n k B^{k-1} A_s \| h_0 \|_{m_0}
\leq \frac{1}{1-B}\, \| h_0 \|_s + \frac{1}{(1-B)^2} \, \frac78 R \rho_s \| h_0 \|_{m_0}.
$
Hence $(i)_{n+1}$ holds by choosing $\rho_s = 2 \| \eta \|_s$. 
The proof of \eqref{i ii} is complete. 

As a consequence, $w_n$ is a Cauchy sequence in $H^s$, and its limit $w$ satisfies 
$\| w \|_s \leq \rho_s = 2 \| \eta \|_s$. 

\emph{Continuity}. 
The inverse map $(\Phi^{(4)})^{-1}$ is Lipschitz-continuous because it is constructed
as a solution of the fixed point problem (recall \eqref{Lip}).
\end{proof}

\begin{lemma}  \label{lemma:3001.1}
For all complex functions $u,v,y,h$, one has 
\begin{alignat}{2} 
\label{A12 C12 self-adj}
\la A_{12}[u,v] y , h \ra 
& = \la y , A_{12}[u,v] h \ra, 
\quad \ \ & 
\la C_{12}[u,v] y , h \ra 
& = \la y , C_{12}[u,v] h \ra,
\\
\label{A12 C12 conj}
\overline{A_{12}[u,v] y} 
& = A_{12}[ \overline{u}, \overline{v} ] \overline{y}, 
\quad & 
\overline{C_{12}[u,v] y} 
& = C_{12}[ \overline{u}, \overline{v} ] \overline{y},
\\
\label{A12 C12 commu}
A_{12} [u,v] \Lm^s y 
& = \Lm^s A_{12}[u,v] y, 
\quad & 
C_{12} [u,v] \Lm^s y 
& = \Lm^s C_{12}[u,v] y
\end{alignat}
where $\overline{u}$ is the complex conjugate of $u$, and so on.
As a consequence, for all complex functions $w,z,y,h$, one has 
\begin{alignat}{2} 
\label{M12 M21 self-adj}
\la M_{12}(w,z) y , h \ra & = \la y , M_{12}(w,z) h \ra, 
\quad &
\la M_{21}(w,z) y , h \ra & = \la y , M_{21}(w,z) h \ra,
\\ 
\label{M12 M21 conj}
\overline{M_{12}(w,z)h} & = M_{12}(\overline{w}, \overline{z}) \overline{h}, 
\quad &
\overline{M_{21}(w,z)h} & = M_{21}(\overline{w}, \overline{z}) \overline{h}, 
\\ 
\label{M12 M21 commu}
[ M_{12}(w,z) , \Lm^s ] & = 0, \quad & 
[ M_{21}(w,z) , \Lm^s ] & = 0.
\end{alignat}
Moreover, for all complex $w,z,h$, 
\begin{equation} \label{M12(f,g)=M21(g,f)}
M_{12}(w,z) h = M_{21}(z,w) h
\end{equation}
and 
\begin{equation} \label{anticommu MD}
M(w,z) \mD_1 + \mD_1 M(w,z) = 0.
\end{equation}
\end{lemma}

\begin{proof} 
All \eqref{A12 C12 self-adj}-\eqref{M12(f,g)=M21(g,f)}
directly follow from the definition \eqref{fix a12 c12} 
of the coefficients $a_{12}(j,k), c_{12}(j,k)$  
and from \eqref{fix A21 C21}, \eqref{formula M}.
The anti-commutator identity \eqref{anticommu MD} follows from \eqref{M12 M21 commu}.
\end{proof}

\begin{lemma} \label{lemma:3001.2}
The maps $M(w, \overline{w})$, $K(w,\overline{w})$, 
and the transformation $\Phi^{(4)}$ 
preserve the structure of real vector field \eqref{real vector field}.
Hence $X^+$ defined in \eqref{def X+} satisfies \eqref{real vector field}.
\end{lemma}

\begin{proof} 
It follows from Lemma \ref{lemma:3001.1}. 
\end{proof}

For a system $\pa_t (w, \overline{w}) = \mF(w, \overline{w})$ 
where the vector field $\mF = (\mF_1, \mF_2)$ satisfies \eqref{real vector field}, 
the Sobolev norm of any solution evolves in time according to
\begin{align} 
\pa_t (\| w \|_s^2) 
& = \la \Lm^s \mF_1(w, \overline{w}) , \Lm^s \overline{w} \ra 
+ \la \Lm^s w , \Lm^s \mF_2(w, \overline{w}) \ra
\notag \\
& = 2 \Re \la \Lm^s \mF_1(w, \overline{w}) , \Lm^s \overline{w} \ra.
\label{generic en est}
\end{align}
The vector field $X^+$ in \eqref{X+} is
\begin{align}
X^+(w,z) 
& = \mD_1(w,z)  
- K(w,z) \mD_1 \Big( M(w,z) \begin{pmatrix} w \\ z \end{pmatrix} \Big)
+ \tilde K(w,z) \mD_1( \Phi^{(4)}(w,z))
\notag \\ & \quad \ 
+ X_3^+(w,z)
+ (I + K(w,z))^{-1} \mD_{\geq 3} ( \Phi^{(4)}(w,z))
+ \mR_{\geq 5}^+(w,z)
\label{X+ bis}
\end{align}
where
\begin{align*}
\mR_{\geq 5}^+(w,z)
& := (I + K(w,z))^{-1} \mR_{\geq 5} ( \Phi^{(4)}(w,z))
+ [ \mB_3 ( \Phi^{(4)}(w,z)) - \mB_3(w,z) ]
\notag \\ & \qquad 
+ \big( - K(w,z) + \tilde K(w,z) \big) \mB_3 ( \Phi^{(4)}(w,z)).
\end{align*}
By \eqref{anticommu MD}, equation \eqref{homological} becomes 
\begin{equation} \label{magic}
\big( M(w,z) + K(w,z) \big) \mD_1 \begin{pmatrix} w \\ z \end{pmatrix} 
= \mB_3(w,z) - X_3^+(w,z).
\end{equation}
We use \eqref{magic} to rewrite the terms in \eqref{X+ bis} containing $\mD_1$, $\mD_{\geq 3}$. 
At a first glance, these terms seem to be unbounded, 
as $\mD_1, \mD_{\geq 3}$ are operators of order one,  
but, using \eqref{magic}, it becomes clear that they are, in fact, bounded.
Omitting to write $\binom{w}{z}$ and $(w,z)$, 
identity \eqref{magic} is $(M+K)\mD_1 = \mB_3 - X_3^+$, 
the anti-commutator formula \eqref{anticommu MD} is $M \mD_1 + \mD_1 M = 0$,
and therefore we have
\begin{align}
- K \mD_1 M + \tilde K \mD_1 (I + M)
& = K M \mD_1 + \sum_{n=2}^\infty (- K)^n \mD_1 
+ \sum_{n=2}^\infty (-K)^n \mD_1 M
\notag \\ 
& = K M \mD_1 + \sum_{n=2}^\infty (- K)^n \mD_1 
- \sum_{n=2}^\infty (-K)^n M \mD_1
\notag \\ 
& = - \sum_{n=1}^\infty (-K)^n M \mD_1 
- \sum_{n=1}^\infty (-K)^n K \mD_1
\notag \\ 
& = - \sum_{n=1}^\infty (-K)^n (M+K) \mD_1
\notag \\
& = K (I+K)^{-1} (\mB_3 - X_3^+).
\label{2nd 3rd after magic}
\end{align}
Regarding the terms with $\mD_{\geq 3}$, recalling \eqref{def mD geq 3} 
one has 
\begin{equation} \label{mD geq 3 composta}
\mD_{\geq 3}(\Phi^{(4)}(w,z)) 
= \mP(w,z) \mD_1 (\Phi^{(4)}(w,z)) 
\end{equation}
where 
\begin{equation} \label{def mP}
\mP(w,z) := \sqrt{1 + 2 P(\Phi^{(4)}(w,z))} \, - 1.
\end{equation}
We recall that $P$ is defined in \eqref{Q ph Q}, \eqref{def Q(u,v)}, \eqref{def ph}, 
and it is a function of $t$ only (i.e., it does not depend on $x$).
We write \eqref{mD geq 3 composta} as
$\mD_{\geq 3}(I+M) = \mP \mD_1(I+M)$, where $\mP$ is the multiplication operator 
$\mP h = \mP(w,z) h$. 
Using the identities $(M+K)\mD_1 = \mB_3 - X_3^+$ and $M \mD_1 + \mD_1 M = 0$, 
namely \eqref{magic} and \eqref{anticommu MD}, 
and the fact that $\mP K = K \mP$ (because $\mP$ is a function of time only), 
we calculate
\begin{align}
(I+K)^{-1} \mD_{\geq 3} \Phi^{(4)}
& = (I+K)^{-1} \mP \mD_1(I+M) 
\notag \\ 
& = \mP \Big( \sum_{n=0}^\infty (-K)^n \mD_1 + \sum_{n=0}^\infty (-K)^n \mD_1 M \Big)
\notag \\ 
& = \mP \mD_1 
- \mP \Big( \sum_{n=0}^\infty (-K)^n K \mD_1 + \sum_{n=0}^\infty (-K)^n M \mD_1 \Big)
\notag \\ 
& = \mP \mD_1 - \mP (I+K)^{-1} (K + M) \mD_1
\notag \\ 
& = \mP \mD_1 - \mP (I+K)^{-1} (\mB_3 - X_3^+).
\label{5th after magic}
\end{align}
By \eqref{2nd 3rd after magic} and \eqref{5th after magic}, 
the vector field $X^+$ in \eqref{X+ bis} becomes 
\begin{equation} \label{X+ ter}
X^+(w,z) 
= \big( 1 + \mP(w,z) \big) \mD_1(w,z) + X_3^+(w,z)
+ X_{\geq 5}^+(w,z)
\end{equation}
where 
\begin{align} 
X_{\geq 5}^+(w,z)
& := K(w,z) \big( I+K(w,z) \big)^{-1} \big( \mB_3(w,z) - X_3^+(w,z) \big)
+ \mR_{\geq 5}^+(w,z)
\notag \\ & \quad \ 
- \mP(w,z) \big( I+K(w,z) \big)^{-1} \big( \mB_3(w,z) - X_3^+(w,z) \big).
\label{def X+ geq 5}
\end{align}

To analyze the energy estimate \eqref{generic en est} for $\mF = X^+$,
we prove that the contribution of $(1 + \mP)\mD_1$ and $X_3^+$ is zero,
and the one of all the other terms is quintic.
Since $\mP = \mP(w,z)$ is a function of time only, one simply has 
\begin{equation*} 
\la \Lm^s (1 + \mP) (- i \Lm w) , \Lm^s z \ra 
+ \la \Lm^s w , \Lm^s (1 + \mP) i \Lm z \ra = 0.
\end{equation*}
Next, recalling \eqref{X3+ 1} and \eqref{X3+ 2}, one has 
\begin{align}
& \la \Lm^s (X_3^+)_1, \Lm^s z \ra + \la \Lm^s w , \Lm^s (X_3^+)_2 \ra 
\notag \\ 
& = - \frac{i}{4} 
\sum_{\begin{subarray}{c} j,k \in \Z^d \setminus \{ 0 \} \\ |k| = |j| \end{subarray}} 
w_j w_{-j} |j|^2 z_k |k|^{2s} z_{-k} 
+ \frac{i}{4} \sum_{\begin{subarray}{c} j,k \in \Z^d \setminus \{ 0 \} \\ |k| = |j| \end{subarray}} 
z_j z_{-j} |j|^2 w_k |k|^{2s} w_{-k}
= 0
\label{below is here}
\end{align}
(rename $j \leftrightarrow k$ in the second sum and use $|j| = |k|$).
To estimate the contribution of $X_{\geq 5}^+$, 
we collect a few elementary estimates in the next lemma.

\begin{lemma} \label{lemma:elementary}
For all $s \geq 0$, all pairs of complex conjugate functions $(w,z)$, 
one has
\begin{equation} \label{stima mB3 X3+}
\| \mB_3(w,z) \|_s \leq \frac12 \| w \|_1^2 \| w \|_s, 
\quad 
\| X_3^+(w,z) \|_s \leq \frac14 \| w \|_1^2 \| w \|_s, 
\end{equation}
and, for $\| w \|_{m_0} \leq \frac12$, for all complex functions $h$,  
\begin{align} \label{stima mP}
\| \mP(w,z) h \|_s & = \mP(w,z) \| h \|_s,
\quad 
0 \leq \mP(w,z) \leq C \| w \|_{\frac12}^2,
\\ 
\label{stima mR geq 5}
\| \mR_{\geq 5}(w,z) \|_s 
& \leq 2 P(w,z) \| \mB_3(w,z) \|_s 
\leq C \| w \|_{\frac12}^2 \| w \|_1^2 \| w \|_s
\end{align}
where $\mR_{\geq 5}$ is defined in \eqref{def mR geq 5}
and $C$ is a universal constant.
\end{lemma}

\begin{proof}
Estimate \eqref{stima mB3 X3+} follows from 
\eqref{def mB}, \eqref{X3+ 1}, \eqref{X3+ 2}. 
To prove \eqref{stima mP} and \eqref{stima mR geq 5}, 
recall \eqref{def mP}, \eqref{Q ph Q}, \eqref{def Q(u,v)}, \eqref{def ph}, 
\eqref{def Phi4}, \eqref{stima M}.
\end{proof}

\begin{lemma} 
For all $s \geq 0$, all $(w,z) \in H^s_0(\T^d, c.c.) \cap 
H^{m_0}_0(\T^d, c.c.)$ with $\| w \|_{m_0} \leq \frac12$, one has 
\begin{equation*} 
\| X_{\geq 5}^+ (w,z) \|_s 
\leq C \| w \|_1^2 \| w \|_{m_0}^2 \| w \|_s
\end{equation*}
where $C$ is a universal constant.
\end{lemma}

\begin{proof}
Use \eqref{def X+ geq 5} and Lemma \ref{lemma:elementary}.
\end{proof}

As a consequence, we obtain the following improved energy estimate.

\begin{lemma} \label{lemma:brand new}
Let $T > 0$, $s \geq m_0$. 
Any solution $(w , \overline{w}) \in C^0([0,T], H^s_0 (\T^d, c.c.))$
of the system $\pa_t (w , \overline{w}) = X^+ (w, \overline{w})$ 
satisfies 
\begin{equation} \label{en est X+}
\pa_t (\| w \|_s^2) \leq  C_* \| w \|_1^2 \| w \|_{m_0}^2 \| w \|_s^2
\end{equation}
as long as it remains in the ball $\| w \|_{m_0} \leq \frac12$, 
for some universal constant $C_* > 0$.
\end{lemma}

\section{Proof of Theorem \ref{thm:main}}
\label{sec:proof}

We now perform the composition of all the changes of variables defined in the previous sections, namely we define
\[
\Phi = \Phi^{(1)} \circ \Phi^{(2)} \circ \Phi^{(3)} \circ \Phi^{(4)},
\]
where $\Phi^{(1)}$, $\Phi^{(2)}$, $\Phi^{(3)}$ and $\Phi^{(4)}$ have been defined in \eqref{1912.2}, \eqref{def Phi2}, \eqref{uv eta psi}, \eqref{def Phi3}, \eqref{def Phi4}. The definitions of $\Phi^{(1)}$ and $\Phi^{(2)}$, together with Lemma \ref{lemma:Phi3 inv} and Lemma \ref{lemma:Phi4 inv}, directly imply the following lemma. 

\begin{lemma}\label{lemma:compo inv}
There exist universal constants $\d_0 \in (0, \frac14)$, $C_0 >0$ such that, 
for all $s \geq m_0$ (where $m_0$ is defined in \eqref{def m0}), 
for all pairs of zero mean real functions 
$(u,v)\in H^{s+\frac12}_0 (\T^d, \R) \times H^{s-\frac12}_0 (\T^d, \R)$ 
satisfying
\begin{equation*} 
\| u \|_{m_0+\frac12} + \| v \|_{m_0-\frac12} \leq \d_0,
\end{equation*}
there exists a unique pair $(w,z)=(w, \overline{w})\in H^s_0(\T^d,c.c.)$ such that 
$(u,v)=\Phi(w, \overline{w})$. 
Moreover, $(w,\overline{w}) =\Phi^{-1}(u,v)$ satisfies the estimate
\begin{equation*} 
\| w \|_s \leq C_0 \big( \| u \|_{s+\frac12} + \| v \|_{s-\frac12} \big).
\end{equation*}
Conversely, if $w\in H^s_0(\T^d, \C)$ satisfies
\begin{equation*} 
\| w \|_{m_0} \leq \d_0,
\end{equation*}
then $(u,v)=\Phi(w,\overline{w})\in H^{s+\frac12}_0 (\T^d, \R) \times H^{s-\frac12}_0 (\T^d, \R)$ 
is a pair of zero mean real functions satisfying
\begin{equation*} 
\| u \|_{s+\frac12} + \| v \|_{s-\frac12} \leq C_0 \| w \|_s.
\end{equation*}
\end{lemma}

As a consequence, in the following corollary we deduce 
the equivalence of the Kirchhoff equation \eqref{K1}
and the transformed system \eqref{def X+}.

\begin{corollary} \label{cor:cor}
Let $\d_0,C_0 >0$ be given by Lemma \ref{lemma:compo inv}. 
Then, for all $s\geq m_0$, if $u\in C^0([0,T], H^{s+\frac12}_0 (\T^d, \R)) 
\cap C^1([0,T], H^{s-\frac12}_0 (\T^d, \R))$ 
is a solution of equation \eqref{K1} on some time interval $[0,T]$ with 
\begin{equation*} 
\max_{t \in [0,T]} \big( \| u (t) \|_{m_0+\frac12} 
+ \| \pa_t u (t)  \|_{m_0-\frac12} \big) \leq \d_0,
\end{equation*}
then the pair $(w,z)=(w, \overline{w})\in C^0 ([0,T], H^s_0(\T^d, c.c.))$ 
defined as $(w(t), \overline{w}(t)) = \Phi^{-1} (u(t),\pa_t u(t))$ 
is a solution of system \eqref{def X+}, satisfying
\begin{equation*} 
\max_{t \in [0,T]} \| w(t) \|_s \leq C_0 \max_{t \in [0,T]} 
\big( \| u (t) \|_{s+\frac12} + \| \pa_t u (t)  \|_{s-\frac12} \big).
\end{equation*}
Conversely, if $(w,\overline{w})\in C^0 ([0,T], H^s_0(\T^d, c.c.))$ 
is a solution of system \eqref{def X+} satisfying
\begin{equation*} 
\max_{t \in [0,T]} \| w(t) \|_{m_0} \leq \d_0,
\end{equation*}
then the pair of real functions $(u,v)\in C^0([0,T], H^{s+\frac12}_0 (\T^d,\R)) 
\times C^0([0,T], H^{s-\frac12}_0 (\T^d, \R))$ 
defined as $(u(t),v(t))=\Phi(w(t),\overline{w}(t))$ satisfies $v=\pa_t u$
and $u\in C^0([0,T], H^{s+\frac12}_0 (\T^d)) \cap C^1([0,T], H^{s-\frac12}_0 (\T^d))$ 
is a solution of equation \eqref{K1} satisfying
\begin{equation*}
\max_{t \in [0,T]} \big( \| u (t) \|_{s+\frac12} + \| \pa_t u(t)  \|_{s-\frac12} \big) 
\leq C_0 \max_{t \in [0,T]} \| w(t) \|_s.
\end{equation*}
\end{corollary}

By a repeated use of Lemma \ref{lemma:compo inv} and Corollary \ref{cor:cor}
we prove Theorem \ref{thm:main}.

\begin{proof}[Proof of Theorem \ref{thm:main}]
The classical local existence and uniqueness theory 
for the Kirchhoff equation \eqref{K1} (see \cite{Arosio Panizzi 1996})
and Corollary \ref{cor:cor} imply the local existence and uniqueness for system \eqref{def X+}, 
for every initial data $(w_0, \overline{w_0})$ in the ball $\| w_0 \|_{m_0} \leq \d_0$. 

Let $(\a, \b) \in H^{m_0 + \frac12}_0(\T^d, \R) \times H^{m_0 - \frac12}_0(\T^d, \R)$
with 
\[
\e := \| \a \|_{m_0 + \frac12} + \| \b \|_{m_0 - \frac12} \leq \e_0 := \frac{\d_0}{2 C_0}.
\]
Let $(w_0, \overline{w_0}) := \Phi^{-1}(\a,\b)$. 
By Lemma \ref{lemma:compo inv}, one has 
$\| w_0 \|_{m_0} \leq C_0 \e \leq \frac{\d_0}{2}$, 
and therefore the Cauchy problem for system \eqref{def X+} 
with initial data $(w_0, \overline{w_0})$ has a (unique) local solution 
$(w(t), \overline{w}(t))$, whose existence time can be extended as long as 
$w(t)$ remains in the ball $\| w \|_{m_0} \leq \d_0$. 
By Lemma \ref{lemma:brand new}, 
\[
\pa_t (\| w(t) \|_{m_0}^2) \leq C_* \| w(t) \|_{m_0}^6.
\]
Hence 
\[
\| w(t) \|_{m_0} 
\leq \frac{\| w_0 \|_{m_0}}{(1 - 2 C_* \| w_0 \|_{m_0}^4 t)^{\frac14}} 
\leq 2 \| w_0 \|_{m_0} \leq \d_0
\]
for all $t \in [0,T]$, with
\[
T := \frac{C_1}{\e^4}, \qquad C_1 := \frac{15}{32 C_* C_0^4}.
\] 
Then $(u,v) := \Phi(w, \overline{w})$ belongs to 
$C^0([0,T], H^{m_0 + \frac12}_0(\T^d,\R)) \times C^0([0,T], H^{m_0 - \frac12}_0(\T^d,\R))$ 
and solves \eqref{p1}, so that 
$u \in C^0([0,T], H^{m_0 + \frac12}_0(\T^d,\R)) \cap C^1([0,T], H^{m_0 - \frac12}_0(\T^d,\R))$ 
solves \eqref{K1} with initial data $(\a,\b)$, and 
$\| u(t) \|_{m_0 + \frac12} + \| \pa_t u(t) \|_{m_0 - \frac12} \leq 2 C_0^2 \e$ 
for all $t \in [0,T]$. 

If, in addition, $(\a,\b) \in H^{s+\frac12}_0(\T^d,\R) \times H^{s-\frac12}_0(\T^d,\R)$ 
for some $s \geq m_0$, then $w_0 \in H^s_0(\T^d,\C)$ and, by Lemma \ref{lemma:brand new}, 
\[
|\pa_t (\| w(t) \|_s^2)| \leq C_* \| w(t) \|_{m_0}^4 \| w(t) \|_s^2
\leq C_* (2 \| w_0 \|_{m_0})^4 \| w(t) \|_s^2 
\]
for all $t \in [0,T]$. Hence  
\[
\| w(t) \|_s \leq \| w_0 \|_s \exp(8 C_* \| w_0 \|_{m_0}^4 t),
\]
whence 
\[
\| u(t) \|_{s+\frac12} + \| \pa_t u(t) \|_{s-\frac12}
\leq C_0^2 e^{\frac{15}{4}} (\| \a \|_{s+\frac12} + \| \b \|_{s-\frac12})
\]
for all $t \in [0,T]$. 
The proof of Theorem \ref{thm:main} is complete.
\end{proof}

\begin{footnotesize}


\noindent
Pietro Baldi, Emanuele Haus 

\noindent
Dipartimento di Matematica e Applicazioni ``R. Caccioppoli'', 
Universit\`a di Napoli Federico II, 
Via Cintia, Monte S. Angelo, 
80126 Napoli, Italy

\noindent
\texttt{pietro.baldi@unina.it}, 
\texttt{emanuele.haus@unina.it}
\end{footnotesize}


%
%
%
%
%
%
%
%
%

\end{document}